\documentclass[12pt]{elsarticle}
\usepackage{cancel}
\usepackage{caption}
\usepackage{color}
\usepackage{lscape}
\usepackage{afterpage}
\usepackage{pstricks}
\usepackage{pst-plot}
\usepackage{longtable}
\usepackage{dcolumn}
\usepackage{pst-node}
\usepackage{amsmath}
\usepackage{amssymb}
\usepackage{amsthm}
\usepackage{multirow}
\usepackage{colortab}
\usepackage{color}
\usepackage{array}
\usepackage{slashbox}
\usepackage{colortbl}
\usepackage{subfigure}
\usepackage{textcomp}
\usepackage{pstricks, pst-node}
\usepackage{pst-all}
\usepackage{algorithm,algorithmic}
\usepackage{url}
\usepackage{soul}
\usepackage{hyperref}
\psset{arrows=->, labelsep=3pt, mnode=circle}

\usepackage{tabularx}

\usepackage{eurosym}

\newfont{\rams}{msbm10 scaled\magstep1}
\newcommand{\rea}{\mbox{\rams \symbol{'122}}}

\newenvironment{resumeT}{\begin{list}{}{\setlength{\rightmargin}{\leftmargin}}\item[]
{\centering {\bf \it~~~}
\par}\item[]\ignorespaces}{\unskip\end{list}}

\newtheorem{prop}{Proposition}[section]

\newtheorem{theorem}{Theorem}[section]

\begin{document}
\title{As simple as possible but not simpler in Multiple Criteria Decision Aiding: the ROR-SMAA level dependent Choquet integral approach}

\author[uc]{Sally Giuseppe Arcidiacono}
\ead{s.arcidiacono@unict.it}
\author[uc]{Salvatore Corrente}
\ead{salvatore.corrente@unict.it}
\author[uc,por]{Salvatore Greco}
\ead{salgreco@unict.it}

\address[uc]{Department of Economics and Business, University of Catania, Corso Italia, 55, 95129  Catania, Italy}
\address[por]{Portsmouth Business School, Centre for Operational Research and Logistics (CORL), University of Portsmouth, Portsmouth, United Kingdom}

\date{}
\maketitle

\vspace{-1cm}

%%%%%%%%%%%%%%%%%%%%%%%%%%%%%%%%%%%%%%%%%%%%%%%%%%%%%%%%%%%%%%%%%%%%%%%%%%%%%%%%%%%%%%%%%%%%%%%%%%%%%%%%%%%%%%%%%%%%%%%%%%%%%%%%%%%%%%%%%
\begin{resumeT}
{\large {\bf Abstract:}}  
The level dependent Choquet integral has been proposed to handle decision making problems in which the importance and the interaction of criteria may depend on the level of the alternatives' evaluations. This integral is based on a level dependent capacity, which is a family of single capacities associated to each level of evaluation for the considered criteria. We present two possible formulations of the level dependent capacity where importance and interaction of criteria are constant inside each one of the subintervals in which the interval of evaluations for considered criteria is split or vary with continuity inside the whole interval of evaluations. Since, in general, there is not only one but many level dependent capacities compatible with the preference information provided by the Decision Maker, we propose to take into account all of them by using the Robust Ordinal Regression (ROR) and the Stochastic Multicriteria Acceptability Analysis (SMAA). On one hand, ROR defines a necessary preference relation (if an alternative $a$ is at least as good as an alternative $b$ for all compatible level dependent capacities), and a possible preference relation (if $a$ is at least as good as $b$ for at least one compatible level dependent capacity). On the other hand, considering a random sampling of compatible level dependent capacities, SMAA gives the probability that each alternative reaches a certain ranking position as well as the probability that an alternative is preferred to another. A real-world decision problem on rankings of universities is provided to illustrate the proposed methodology.

%\vspace{1cm}
%%%%%%%%%%%%%%%%%%%%%%%%%
%%%%%%% Keywords %%%%%%%%
%%%%%%%%%%%%%%%%%%%%%%%%%

\vspace{0,3cm}
\noindent{\bf Keywords}: {Decision support systems, Level dependent Choquet integral, Robust Ordinal Regression, Stochastic Multicriteria Acceptability Analysis, Ranking of University}
\end{resumeT}

\pagenumbering{arabic}

%%%%%%%%%%%%%%%%%%%%%%%%%%%%%%%%%%%%%%
\section{Introduction}\label{intro}%%%
%%%%%%%%%%%%%%%%%%%%%%%%%%%%%%%%%%%%%%
Decision support is generally based on formal mathematical models representing the preferences of the Decision Maker (DM). Since all formal mathematical models are a simplified representation of the phenomena one wants to investigate, a decision support model is a simplified representation of a process that the DM can adopt to take a sensible decision. The idea of simplicity of a model is often related to the Ockham's razor, usually expressed as ``\textit{pluralitas non est ponenda sine necessitate}" (plurality should not be posited without necessity) even if in his ``Summa Totius Logicae'',  William of Ockham wrote in a different way: ``\textit{Frustra fit per plura quod potest fieri per pauciora}" (it is futile to do with more things that which can be done with fewer) \cite{KnealeKneale1962,Thorburn1918}. The search of simplicity is, for sure, an important point of any scientific model. However it has to be tempered with the idea that the model should represent all the important elements that can explain the phenomenon under investigation. In this sense, there is another very well known principle that is generally attributed to Albert Einstein: ``\textit{Everything should be made as simple as possible, but not simpler}". In fact, the exact expression of Einstein is the following: ``\textit{It can scarcely be denied that the supreme goal of all theory is to make the irreducible basic elements as simple and as few as possible without having to surrender the adequate representation of a single datum of experience}" \cite{einstein1934method}. We believe that this principle has a specific importance for decision support models. They aim to provide a recommendation that has to be understandable and convincing for the DM and not an ideally abstract description of the decision processes as can be the case of the decision models considered in economics \cite{roy1993decision}. Indeed, in economics, following the considerations of Milton Friedman \cite{friedman1953essays},  a model has to be evaluated in terms of simplicity (that is minimality of initial knowledge needed) and fruitfulness (that is accuracy and generality of the obtained prediction). Instead, in decision support, a model has to be appreciated in terms of the interaction it permits with the DM and in terms of the convincing arguments it produces \cite{RoyRobustness}. With this aim, the focus shifts from the simplicity to the richness of the model that, within decision support, is seen as a feature that facilitates the interaction with the DM that can see the model as a language to express his preferences with the desired degree of detail. 

On the basis of these considerations, in this paper we propose  a quite ``rich'' decision support model that can be seen as a relevant application and an exemplification of the tempered parsimony we have discussed. Our model is constituted by the Robust Ordinal Regression (ROR; see \cite{greco2008ordinal} for the seminal paper and \cite{CGKSml,CGKSen} for two comprehensive  surveys) and Stochastic Multicriteria Acceptability Analysis (SMAA; \cite{Lahdelma} for the seminal paper and \cite{LahdelmaSalminenbook,PelissariEtAl2019} for two comprehensive  surveys) applied to the level dependent Choquet integral \cite{greco2011choquet}. The decision analysis model so obtained permits to take into consideration the three following relevant aspects of a multicriteria decision making problem:  
\begin{itemize}
	\item interaction between criteria, 
	\item importance and interaction of criteria changing from one level to another of the evaluations in the considered criteria,
	\item robustness concerns with respect to the value assigned to the parameters of the decision model.
\end{itemize}
	Let us discuss more in detail each one of the above points. Interaction between criteria refers to the phenomenon for which when aggregating their evaluations to obtain an overall assessment, criteria can reinforce or can weaken each other producing a synergy or a redundancy, respectively; following a very well known example \cite{Grabisch1996}, evaluating students of a technical school we can have: 
\begin{itemize}
\item a redundancy effect between Mathematics and Physics: students good in Mathematics are, in general, good also in Physics. Therefore, to avoid an over-evaluation of students good in both subjects, Mathematics and Physics considered together should have a weight smaller than the sum of their weights when considered separately;
\item a synergy effect between Mathematics and Literature: students good in Mathematics are, in general, not good in Literature. Therefore, a student presenting good marks in both subjects is well appreciated and a special bonus is considered in evaluating him. Consequently, Mathematics and Literature considered together should have a weight greater than the sum of their weights when considered separately.
\end{itemize}
We have to observe also that importance and interaction of criteria can change from one level to another of the evaluations in the considered criteria; again, in evaluating students in a technical school, we can have the following:
			
\begin{itemize}			
\item the importance of Literature can be greater than that one of Mathematics and Physics in evaluating students with rather low notes on the last two subjects: indeed, a relatively good note in Literature means that there is some potential for the student to improve despite the low notes in Mathematics and Physics,
\item the importance of Literature is not so great in evaluating students with medium notes in all subjects: in this case, the focus is more on the core subjects of the school that are Mathematics and Physics,
\item the importance of Literature is again very high in evaluating students with high notes in all subjects: in this case, one expects excellence and this is obtained if the students are good in all subjects and not only in Mathematics and Physics;
\item Mathematics and Physics can present a synergetic effect in evaluating students with low notes: indeed, in this case, a relatively good note in both subjects means that the student can improve;
\item Mathematics and Physics can present a redundant effect in evaluating students with medium notes: in this case, one relatively good evaluation in one subject is already considered enough;
\item Mathematics and Physics can present a synergetic effect in evaluating students with high notes in both subjects: in this case, good evaluations are necessary to be evaluated excellent.
\end{itemize}
			 			
Another important issue in any decision aiding procedure is the robustness of the final recommendation. More in detail, any formal decision model requires fixing the value of a certain number of parameters. Of course, perturbation of the values of these parameters can bring to different results. Consequently, to get a sufficiently reliable recommendation, it is necessary to explore the stability of the suggestions proposed by the adopted decision model changing the values of its parameters.

The decision model we are considering is able to take into account the above three issues in the way we are going to present in the following. 

On one hand, the level dependent Choquet integral permits to handle criteria for which the importance and the interaction change from one level to the other of the evaluations. Indeed, the Choquet integral permits to aggregate the evaluations taking into account interaction of criteria by using capacities that assign a weight to each subset of the considered criteria and not only to single criteria as in the usual weighted sum. The level dependent Choquet integral goes a further step forward taking into account a set of level dependent capacities, that is,  a capacity for each level of the evaluation criteria. This permits to represent different attitudes, that is, different importance and interaction assigned to considered criteria, for different evaluation levels (for an alternative decision model permitting to represent interaction of criteria that changes from one level to another of the evaluations see \cite{greco2014robust}). 

On the other hand, the robustness concerns are handled by an exploration of the space of the feasible parameters of the decision model, that are, the level dependent capacities compatible with the preference information supplied by the DM. This is done by applying:  
\begin{itemize}
\item ROR, that, for each pair of alternatives $a$ and $b$, tells us if $a$ is preferred to $b$ for all compatible level dependent capacities (in which case, we say that $a$ is \textit{necessarily} preferred to $b$) or for at least one compatible level dependent capacity (in which case we say that $a$ is \textit{possibly} preferred to $b$; 
\item SMAA, that, for each alternative $a$ and for each rank order position $r$,  tells us which is the probability that, randomly extracting a compatible level dependent capacity, $a$ attains the position $r$, as well as, for each pair of alternatives  $a$ and $b$, the probability that $a$ is preferred to $b$.  
\end{itemize}
Another interesting issue to be considered is the modeling of changes of importance and interaction of criteria from one level to another of the evaluations on considered criteria. In fact, two main approaches could be considered:
\begin{itemize}
	\item partitioning in intervals separated by reference levels range of possible evaluations on considered criteria, adopting preference parameters representing importance and interaction between criteria that change from one interval to another, but that remain constant within the same interval;
	\item using preference parameters representing importance and interaction between criteria that change continuously from one level of evaluations to another without jumps.
\end{itemize}
The first approach is clearly more manageable and could be considered a reasonable assumption to model preferences. In this way, the model can be refined by increasing the number of intervals of the partition and, therefore, representing better the changes of importance and interaction of criteria from one level of evaluation to another. This type of preference representation can be formulated in terms of  Choquet integral based on interval level dependent capacities proposed in \cite{greco2011choquet}. \\
The second approach could seem more intuitive: it is reasonable to imagine that the importance and interaction of criteria change smoothly without sudden jumps from one level of evaluation to another. Ideally, a model of this type has a set of preference parameters modeling  importance and interaction of criteria for each level of evaluations. This is clearly a problem from the point of view of induction of these parameters from preference information supplied by the DM and their effective use in decision support. Therefore, if one wants that importance and interaction of criteria change continuously, he has to consider some model that represents these changes taking into account a not so big number of parameters that can be induced and handled with a quite manageable procedure, for example, by linear programming. With this aim, we propose the Choquet integral based on piecewise linear level dependent capacities. By using this model, the parameters representing importance and interaction between criteria at each level of evaluation can be determined as linear interpolation of the values that they assume in specific reference levels. We shall show that also the Choquet integral can be formulated as a linear function of the values assumed by the preference parameters in the reference levels, so that linear programming techniques and sampling procedures can still be used in applying ordinal regression, ROR and SMAA. Also this Choquet integral based on piecewise linear level dependent capacities that we are proposing to handle continuous changes in the preference parameters, must be seen in the context of tempered parsimony of multiple criteria decision aiding models that we are advocating. In synthetic terms, we can say that this model can be seen as the exemplification of a principle dual to the above mentioned Einstein  principle of parsimony: ``\textit{Everything should be made as realistic as possible, but not more complex}".  

The paper is organized as follows. In the next section, we present an introductory example that, step by step, informally introduces the model with its potentialities. In the third section, we recall the Choquet integral as well as the level dependent Choquet integral based on interval level dependent capacities, while, in the fourth section, we present the application of ROR and SMAA to this decision model. The fifth section presents the Choquet integral based on piecewise linear capacities permitting to represent continuous changes in importance and interaction of criteria. In the sixth section, a real world application of the level dependent Choquet integral to the ranking of universities is described. The last section collects conclusions and perspectives for future developments. 

%%%%%%%%%%%%%%%%%%%%%%%%%%%%%%%%%%%%%%%%%%%%%%%%%
\section{An introductory example}\label{ainex}%%% 
%%%%%%%%%%%%%%%%%%%%%%%%%%%%%%%%%%%%%%%%%%%%%%%%%
Let us explain the richness and usefulness of the level dependent Choquet integral coupled with methodologies taking into account robustness concerns, namely ROR and SMAA, and representing quite articulated preferences of a DM. Ideally this example illustrates the basic ideas of the tempered parsimony we are postulating, so that it is presented with a certain degree of detail. Inspired by \cite{Grabisch1996}, let us suppose that the dean of a scientifically oriented high school wants to provide an overall evaluation of its students. For the sake of simplicity, the dean wants to take into account only notes on Mathematics and Physics, expressed on a thirty point scale. The dean expresses the following convictions:
\begin{itemize}
	\item considering the required skills, Physics is more important than Mathematics in evaluating students with good notes on both subjects, that are those having a mark of at least 26 in these subjects;
	\item considering students needing to improve their background, that are students having a mark at most equal to 25 on the two subjects, Mathematics is more important than Physics;
	\item  a student with notes at least equal to 26 in both subjects is excellent if he has very good notes both in Mathematics and in Physics;
	\item among students with notes at most equal to 25 in both subjects, the dean prefers students with a relatively good note in at least one of the subjects, because this can be seen as a starting point to reduce the gap from better students.   
\end{itemize}

On the basis of the above considerations,  taking into account the students whose evaluations on the two subjects are presented in Table \ref{EvaluationsMatrixA}, the dean expresses the following preferences

$$
A \succ C \succ B \succ E \succ F \succ D
$$

\noindent where $ X \succ Y$ means that student $X$ is preferred to student $Y$, with $X,Y \in \{A,B,C,D,E,F\}$. 

\begin{table}[htbp]
  \centering
  \caption{Students' notes on Mathematics and Physics\label{EvaluationsMatrixA}}
    \begin{tabular}{l|cc}
          & Mathematics (M) & Physics (Ph) \\
	  \hline
		\hline
  $A$ & 28 & 28 \\
	$B$ & 30 & 26 \\
	$C$ & 26 & 30 \\
	$D$ & 23 & 23 \\
	$E$ & 25 & 21 \\
	$F$ & 21 & 25 \\
	    \end{tabular}%
\end{table}%

One can see that the dean is giving a rich preference information, but one can also wonder if it is possible to represent his preferences using a model as simple as possible. The simplest model one can consider is the weighted sum, that is 

\begin{equation}\label{WS}
U(X)=w_M\cdot M(X)+w_{Ph}\cdot Ph(X)
\end{equation}
where 
\begin{itemize}
	\item $U(X)$ is the overall evaluation of $X$, 
	\item $M(X)$ and $Ph(X)$ are the notes of $X$ in Mathematics and Physics, respectively,
	\item $w_M$ and $w_{Ph}$ are the weights of Mathematics and Physics and they are such that, $w_M \geqslant 0$, $w_{Ph} \geqslant 0$ and $w_M + w_{Ph} = 1$.
\end{itemize}

Apparently, a so simple model cannot model the preference information supplied by the dean. On one hand, taking into account students $B$ and $C$ and reminding that the dean prefers $C$ over $B$, one gets, 

\begin{equation*}
U(B) = w_{M} \cdot 30 + w_{Ph} \cdot 26 < w_{M} \cdot 26 + w_{Ph} \cdot 30 = U(C)
\end{equation*}

\noindent implying that $w_{M}<w_{Ph}$. On the other hand, considering $E$ and $F$ and reminding that the dean prefers $E$ over $F$, one gets

\begin{equation*}
U(E) = w_{M} \cdot 25 + w_{Ph} \cdot 21 > w_{M} \cdot 21 + w_{Ph} \cdot 25 = U(F)
\end{equation*}
that implies $w_{M}>w_{Ph}$. Of course the two inequalities translating the preferences of the dean are incompatible and, consequently, the weighted sum is not able to represent these preferences. 
  
Looking for an amendment of the weighted sum able to represent preferences that cannot be expressed by using (\ref{WS}), one could imagine that weights depend on the level of the notes. For example, one can imagine that Mathematics has a weight $w^1_M$ for notes up to 25, that is $M^1(X)=min\{M(X),25\}$, and weight $w^2_M$ for the possible excess to 25, that is, $M^2(X)=\max\{M(X)-25,0\}$. Analogously, one could assume that Physics has weight $w^1_{Ph}$ for notes up to 25, that is, $Ph^1(X)=\min\{Ph(X),25\}$, and weight $w^2_{Ph}$ for the possible excess to 25, that is, $Ph^2(X)=\max\{Ph(X)-25,0\}$. Nonnegativity and normalization of the weights in this case would require that 

\begin{equation*}
w^1_M, w^2_M, w^1_{Ph}, w^2_{Ph} \geqslant 0, \;w^1_M + w^1_{Ph}=1,\; w^2_M + w^2_{Ph}=1.  
\end{equation*}

\noindent In this way, we could have a new model, the \textit{level dependent weighted sum}, (in fact, the two additive level dependent weighted sum), having the following form:
\begin{equation*}
U_{LDWS}(X)=w^1_M\cdot M^1(X)+w^2_M\cdot M^2(X)+w^1_{Ph}\cdot Ph^1(X)+w^2_{Ph}\cdot Ph^2(X)
\end{equation*}
\noindent that is 
\begin{equation*}
U_{LDWS}(X)=w^1_M\cdot\min\{M(X),25\}+w^2_M\cdot\max\{M(X)-25,0\}+w^1_{Ph}\cdot\min\{Ph(X),25\}+
\end{equation*}
$$
+w^2_{Ph}\cdot\max\{Ph(X)-25,0\}.
$$
The level dependent weighted sum permits to take into account that, for notes at least equal to 26, Physics is more important than Mathematics ($w^2_{Ph} > w^2_{M}$), while for notes at most equal to 25, it is exactly the opposite ($w^1_M > w^1_{Ph}$). 

The level dependent weighted sum is therefore able to represent some of the preferences provided by the dean. Indeed, considering students $B, C, E$ and $F$, the preference of $C$ over $B$ is translated to the constraint 
\begin{equation*}
U_{LDWS}(B)=w^1_M\cdot\min\{30,25\}+w^2_M\cdot\max\{30-25,0\}+w^1_{Ph}\cdot\min\{26,25\}+w^2_{Ph}\cdot\max\{26-25,0\}
\end{equation*}
$$
<
$$
\begin{equation}\label{LDWSBC}
w^1_M\cdot\min\{26,25\}+w^2_M\cdot\max\{26-25,0\}+w^1_{Ph}\cdot\min\{30,25\}+w^2_{Ph}\cdot\max\{30-25,0\}=U_{LDWS}(C)
\end{equation}
\noindent while, the preference of $E$ over $F$ is translated to the constraint 
\begin{equation*}
U_{LDWS}(E)=w^1_M\cdot\min\{25,25\}+w^2_M\cdot\max\{25-25,0\}+w^1_{Ph}\cdot\min\{21,25\}+w^2_{Ph}\cdot\max\{21-25,0\}
\end{equation*}
$$
>
$$
\begin{equation}\label{LDWSEF}
w^1_M\cdot\min\{21,25\}+w^2_M \cdot\max\{21-25,0\}+w^1_{Ph}\cdot\min\{25,25\}+w^2_{Ph}\cdot\max\{25-25,0\}=U_{LDWS}(F).
\end{equation}
On one hand, by (\ref{LDWSBC}) one gets
\begin{equation*}
w^1_M \cdot 25+ w^2_M \cdot 5 +w^1_{Ph} \cdot 25+w^2_{Ph} \cdot 1 < w^1_M \cdot 25+ w^2_M \cdot 1 +w^1_{Ph} \cdot 25+w^2_{Ph} \cdot 5
\end{equation*}
and, consequently, $w^{2}_{M}<w^{2}_{Ph}$.

On the other hand, by (\ref{LDWSEF}) one gets
\begin{equation*}
w^1_M \cdot 25+ w^2_M \cdot 0 +w^1_{Ph} \cdot 21+w^2_{Ph} \cdot 0 > w^1_M \cdot 21+ w^2_M \cdot 0 +w^1_{Ph} \cdot 25+w^2_{Ph} \cdot 0,
\end{equation*}
\noindent and, consequently, $w^{1}_{M}>w^{1}_{Ph}$.  

Therefore, we have seen that the level dependent weighted sum permits to represent importance of criteria changing from one level to another. However, we have still to check if $U_{LDWS}$ allows to consider also the conjunctive or disjunctive nature of the overall evaluation that changes from one level to another. In fact, according with the preference information supplied by the dean, the aggregation procedure has to be: 
\begin{itemize}
\item conjunctive for notes at least equal to 26 in both subjects: in this case, a student good in both subjects is more appreciated than a student very good in one subject and not so good in the other,
\item disjunctive for notes at most equal to 25 in both subjects: in this case, a student relatively good in at least one of the two subjects is more appreciated than a student presenting medium marks on both subjects.
\end{itemize}

In consequence of the previous observations, in evaluating students $A$, $B$ and $C$, the dean stated that $A$ is preferred to $C$ that, in turn, is preferred to $B$. Using the $U_{LDWS}$, the previous preferences are translated into the constraints 
$$
U_{LDWS}(A)>U_{LDWS}(C)>U_{LDWS}(B).
$$
 
\noindent The first inequality, that is, $U_{LDWS}(A)>U_{LDWS}(C)$, implies that

$$
w^1_M\cdot\min\{28,25\}+w^2_M\cdot\max\{28-25,0\}+w^1_{Ph}\cdot\min\{28,25\}+w^2_{Ph}\cdot\max\{28-25,0\}
$$
$$
>
$$
$$
w^1_M\cdot\min\{26,25\}+w^2_M\cdot\max\{26-25,0\}+w^1_{Ph}\cdot\min\{30,25\}+w^2_{Ph}\cdot\max\{30-25,0\}
$$

\noindent from which, $w^{2}_{M}>w^{2}_{Ph}$.  

\noindent The second inequality, that is $U_{LDWS}(C)>U_{LDWS}(B)$, instead, implies that 

$$
w^1_M\cdot\min\{26,25\}+w^2_M\cdot\max\{26-25,0\}+w^1_{Ph}\cdot\min\{30,25\}+w^2_{Ph}\cdot\max\{30-25,0\}
$$
$$
>
$$
$$
w^1_M\cdot\min\{30,25\}+w^2_M\cdot\max\{30-25,0\}+w^1_{Ph}\cdot\min\{26,25\}+w^2_{Ph}\cdot\max\{26-25,0\}
$$

\noindent and, consequently, $w^{2}_{Ph}>w^{2}_{M}$. Of course, the two obtained inequalities are incompatible since $w^{2}_{Ph}$ cannot be simultaneously lower and greater than $w^{2}_{M}$.  

Analogously, taking into consideration students $D, E$ and $F$, and reminding that the dean stated that $E$ is preferred to $F$ that, in turn, is preferred to $D$, the application of the level dependent weighted sum would translate these preferences into the following chain of constraints

$$
U_{LDWS}(E)>U_{LDWS}(F)>U_{LDWS}(D). 
$$

On one hand, the first inequality, that is $U_{LDWS}(E)>U_{LDWS}(F)$, implies $w^1_M > w^1_{Ph}$, while, on the other hand, the second inequality, that is $U_{LDWS}(F)>U_{LDWS}(D)$, implies that $w^1_M < w^1_{Ph}$. Again, the two inequalities are incompatible and, consequently, even the level dependent weighted sum is not able to represent the preferences given by the DM. 

To represent the more or less conjunctive character of the aggregation procedure one can use the Choquet integral \cite{choquet1953theory} being a generalization of the weighted sum that has been successfully applied in multiple criteria decision analysis to handle interaction between criteria. The Choquet integral of the evaluations of student $X$ on Mathematics and Physics can be defined as follows:
\begin{equation*}
U_{Ch}(X)=
\left\{
\begin{array}{lll}
M(X)+w^{Ch}_{Ph} \cdot (Ph(X) - M(X) ) & \mbox{iff} & M(X)\leqslant Ph(x),\\[0,5mm]
Ph(X)+w^{Ch}_{M} \cdot (M(X) - Ph(X) ) & \mbox{iff} & M(X)\geqslant Ph(x),\\[0,5mm]
\end{array}
\right.
\end{equation*}
with  $0 \leqslant w^{Ch}_{M}, w^{Ch}_{Ph} \leqslant 1$ and not necessarily $w^{Ch}_{M} + w^{Ch}_{Ph} =1$. Let us see if the Choquet integral permits to represent the preferences expressed by the dean on students $A, B$ and $C$. Since $A \succ C \succ B$, we should have
$$
U_{Ch}(A)>U_{Ch}(C)>U_{Ch}(B),
$$
and, consequently,

$$
28+w^{Ch}_{Ph} \cdot 0 > 26+4\cdot w^{Ch}_{Ph}>26+4\cdot w^{Ch}_{M}
$$

\noindent from which 
\begin{equation}\label{ChoquetStudrankABCcond}
0.5 > w^{Ch}_{Ph}  > w^{Ch}_{M} 
\end{equation}
which is in agreement with the greater importance assigned by the dean to Physics over Mathematics in case of notes not smaller than 26 in both subjects. Thus we proved that the Choquet integral can represent the preferences about students $A, B$ and $C$.

With analogous considerations, we can show that the Choquet integral can represent also the preferences over students $D, E$ and $F$.   In fact, from $E \succ F \succ D$ we obtain
$$
U_{Ch}(E)>U_{Ch}(F)>U_{Ch}(D)
$$
\noindent and, consequently,
$$
21+4\cdot w^{Ch}_{M} > 21+4\cdot w^{Ch}_{Ph} > 23+0\cdot w^{Ch}_{Ph}.
$$
The inequalities above imply 
\begin{equation}\label{ChoquetStudrankcondEFD}
w^{Ch}_{M}  > w^{Ch}_{Ph} > 0.5 
\end{equation}
which are in agreement with the greater importance assigned by the dean to  Mathematics over Physics in case of notes not greater than 25 in both subjects. Thus also the preferences between the students $E, F$ and $D$ can be represented by the Choquet integral $U_{Ch}$.

However, let us remark that conditions (\ref{ChoquetStudrankABCcond}) and  (\ref{ChoquetStudrankcondEFD}) are incompatible, which means that the Choquet integral is not able to represent, simultaneously, all the preferences expressed by the dean over the six students. The problem is that the importance of criteria as well as the conjunctive or disjunctive nature of the aggregation procedure are different for students with notes at least equal to 26 and students with notes at most equal to 25. Thus, one has to further generalize the model considering the level dependent Choquet integral \cite{greco2011choquet} that permits to take into account importance of criteria and interaction between criteria related to conjunctive or disjunctive nature of the aggregation procedure that can change from one level to the others of the evaluations on considered criteria. With respect to the overall evaluation of students on Mathematics and Physics, the level dependent Choquet integral can be formulated as follows:

\begin{equation}\label{ChoqLevDep}
U_{LDCh}(X)=
\left\{
\begin{array}{lll}
\overline{U}_{M}(X) & \mbox{iff} & M(X)\leqslant Ph(X),\\[0,4mm]
\overline{U}_{Ph}(X) & \mbox{iff} & M(X)\geqslant Ph(X),\\[0,4mm]
\end{array}
\right.
\end{equation}
\noindent where
$$
\overline{U}_{M}(X)=M(X)+w_{Ph}^{Ch,1} \cdot (\min\{Ph(X),25\} - \min\{M(X),25\})+w_{Ph}^{Ch,2} \cdot (\max\{Ph(X),25\}-\max\{M(X),25\})
$$ 
\noindent and
$$
\overline{U}_{Ph}(X)=Ph(X)+w_{M}^{Ch,1} \cdot (\min\{M(X),25\} - \min\{Ph(X),25\})+w_{M}^{Ch,2} \cdot (\max\{M(X),25\}-\max\{Ph(X),25\}).
$$

Let us check if the level dependent Choquet integral is able to represent the preferences over the six students, that is if there exist $w_{M}^{Ch,1}$, $w_{M}^{Ch,2}$, $w_{Ph}^{Ch,1}$ and $w_{Ph}^{Ch,2}$ such that
$$
U_{LDCh}(A) > U_{LDCh}(C) > U_{LDCh}(B) >U_{LDCh}(E) >U_{LDCh}(F) >U_{LDCh}(D).
$$
Using (\ref{ChoqLevDep}), the previous inequalities are translated into the following chain of constraints: 
$$
U_{LDCh}(A)=28 +w_{Ph}^{Ch,1} \cdot (\min\{28,25\} - \min\{28,25\})+w_{Ph}^{Ch,2} \cdot (\max\{28,25\}-\max\{28,25\})
$$
$$
>
$$
$$
U_{LDCh}(C)=26 +w_{Ph}^{Ch,1} \cdot (\min\{30,25\} - \min\{26,25\})+w_{Ph}^{Ch,2} \cdot (\max\{30,25\}-\max\{26,25\})
$$
$$
>
$$
$$
U_{LDCh}(B)=26 +w_{M}^{Ch,1} \cdot (\min\{30,25\} - \min\{26,25\})+w_{M}^{Ch,2} \cdot (\max\{30,25\}-\max\{26,25\})
$$
$$
>
$$
$$
U_{LDCh}(E)=21 +w_{M}^{Ch,1} \cdot (\min\{25,25\} - \min\{21,25\})+w_{M}^{Ch,2} \cdot (\max\{25,25\}-\max\{21,25\})
$$
$$
>
$$
$$
U_{LDCh}(F)=21 +w_{Ph}^{Ch,1} \cdot (\min\{25,25\} - \min\{21,25\})+w_{Ph}^{Ch,2} \cdot (\max\{25,25\}-\max\{21,25\})
$$
$$
>
$$
$$
U_{LDCh}(D)=23 +w_{Ph}^{Ch,1} \cdot (\min\{23,25\} - \min\{23,25\})+w_{Ph}^{Ch,2} \cdot (\max\{23,25\}-\max\{23,25\}).
$$
We obtain the following conditions:
\begin{enumerate}
\item $U_{LDCh}(A) > U_{LDCh}(C)$ $\Rightarrow$ $2 > w^{Ch,2}_{Ph} \cdot 4$ $\Rightarrow$ $0.5 > w^{Ch,2}_{Ph}$,
\item $U_{LDCh}(C) > U_{LDCh}(B)$ $\Rightarrow$ $w^{Ch,2}_{Ph} \cdot 4 > w^{Ch,2}_{M} \cdot 4$ $\Rightarrow$ $w^{Ch,2}_{Ph} > w^{Ch,2}_{M}$,
\item $U_{LDCh}(E) > U_{LDCh}(F)$ $\Rightarrow$ $w^{Ch,1}_{M} \cdot 4 > w^{Ch,1}_{Ph} \cdot 4$ $\Rightarrow$ $w^{Ch,1}_{M} > w^{Ch,1}_{Ph}$,
\item $U_{LDCh}(F) > U_{LDCh}(D)$ $\Rightarrow$ $w^{Ch,1}_{Ph} \cdot 4 > 2$ $\Rightarrow$ $w^{Ch,1}_{Ph} > 0.5$. 
\end{enumerate}
The four inequalities above are not in contradiction and, therefore, the level dependent Choquet integral permits to represent the preferences expressed by the dean. Moreover, the conditions perfectly translate the preferences of the dean in terms of importance of criteria and the desired conjunctive and disjunctive nature of the aggregation. However, there is still another aspect that deserves to be investigated. Let us consider the new students $G$, $H$ and $I$ whose notes in Mathematics and Physics are presented in Table \ref{EvaluationsMatrixB} and let us try to rank them on the basis of conditions 1.-4. above. 

\begin{table}[!h]
  \centering
  \caption{Notes on Mathematics and Physics of the new students\label{EvaluationsMatrixB}}
    \begin{tabular}{l|cc}
          & Mathematics (M) & Physics (Ph) \\
	  \hline
		\hline
  $G$ & 26 & 29 \\
	$H$ & 29 & 26 \\
	$I$ & 30 & 27 \\
	    \end{tabular}%
\end{table}

\noindent On one hand, comparing $G$ and $H$, we get that 
$$
U_{LDCh}(G)=26 +w_{Ph}^{Ch,1} \cdot (\min\{29,25\} - \min\{26,25\})+w_{Ph}^{Ch,2} \cdot (\max\{29,25\}-\max\{26,25\})
$$
$$
>
$$
$$
26 +w_{M}^{Ch,1} \cdot (\min\{29,25\} - \min\{26,25\})+w_{M}^{Ch,2} \cdot (\max\{29,25\}-\max\{26,25\})=U_{LDCh}(H)
$$
\noindent being always true since it coincides with condition 2. \\
On the other hand, comparing $G$ and $I$, we can have $U_{LDCh}(G)>U_{LDCh}(I)$ as well as $U_{LDCh}(I)>U_{LDCh}(G)$, depending on the values of $w_{M}^{Ch,1}$, $w_{M}^{Ch,2}$, $w_{Ph}^{Ch,1}$ and $w_{Ph}^{Ch,2}$. Indeed, we get
$$
U_{LDCh}(G)=26 +w_{Ph}^{Ch,1} \cdot (\min\{29,25\} - \min\{26,25\})+w_{Ph}^{Ch,2} \cdot (\max\{29,25\}-\max\{26,25\})
$$
$$
>
$$
$$
27 +w_{M}^{Ch,1} \cdot (\min\{30,25\} - \min\{27,25\})+w_{M}^{Ch,2} \cdot (\max\{30,25\}-\max\{27,25\})=U_{LDCh}(I),
$$
\noindent iff 
\begin{equation*}
26+w^{Ch,2}_{Ph} \cdot 3 > 27 + w^{Ch,2}_{M} \cdot 3, 
\end{equation*}
as well as $U_{LDCh}(G)<U_{LDCh}(I)$ iff
\begin{equation*}
26+w^{Ch,2}_{Ph} \cdot 3 < 27 + w^{Ch,2}_{M} \cdot 3.
\end{equation*}
None of the two inequalities is in contradiction with conditions 1.-4. Therefore, we have to conclude that the preference information given by the DM, and formally expressed by conditions 1.-4., does not permit to rank order unambiguously all students since different values of parameters can provide a different comparison of $G$ and $I$. For example, taking $w^{Ch,2}_{M}=\frac{1}{18}$ and $w^{Ch,2}_{Ph}=\frac{4}{9}$ we get 
\begin{equation*}
U_{LDCh}(G)=27.333 > 27.166=U_{LDCh}(I) 
\end{equation*}
while, taking $w^{Ch,2}_{M}=\frac{1}{6}$ and $w^{Ch,2}_{Ph}=\frac{1}{3}$ we get 
\begin{equation*}
U_{LDCh}(G)=27 < 27.5=U_{LDCh}(I).
\end{equation*} 
With a greater level of detail, we can say that student $G$ will be overall evaluated at least as good as student $I$, that is, $U_{LDCh}(G) \geqslant U_{LDCh}(I)$, in case  $w^{Ch,2}_{M}$ and  $w^{Ch,2}_{Ph}$ satisfy conditions 1.-4. with the addition of the following one 
\begin{equation*}
U_{LDCh}(G) \geqslant U_{LDCh}(I) \Leftrightarrow
26+w^{Ch,2}_{Ph} \cdot 3 \geqslant 27 + w^{Ch,2}_{M} \cdot 3 \Leftrightarrow w^{Ch,2}_{Ph} \cdot 3 \geqslant 1 + w^{Ch,2}_{M}\cdot 3.
\end{equation*}

\begin{figure}[!h]
\begin{center}
\caption{Representation of the subset of weights in the $w^{Ch,2}_{M}$-$w^{Ch,2}_{Ph}$ space for which $G$ is preferred to $I$ (the triangle $W_{G\succsim I}\equiv P_2P_3P_4$) and $I$ is preferred to $G$ (the trapezoid $W_{I\succsim G}\equiv P_1P_2P_4O$) \label{FigureLD}}
\includegraphics[scale=0.3]{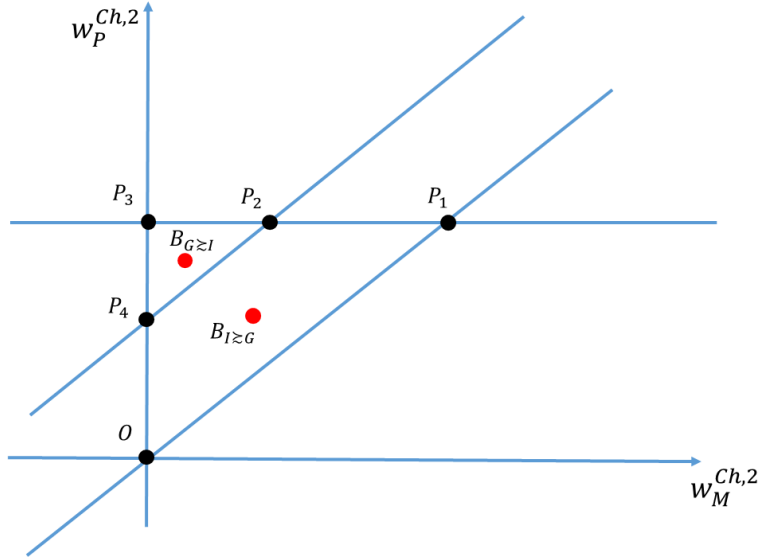}
\end{center}
\end{figure}

In the space $w^{Ch,2}_{M}-w^{Ch,2}_{Ph}$, the preference parameters compatible with this piece of preference information are represented by the points in the triangle $W_{G\succsim I}$ having vertices $P_2\equiv(\frac{1}{6},\frac{1}{2})$, $P_3\equiv(0,\frac{1}{2})$ and $P_4\equiv(0,\frac{1}{3})$ in Figure \ref{FigureLD}. Instead, student $I$ will get an overall evaluation not worse than $G$, that is, $U_{LDCh}(I) \geqslant U_{LDCh}(G)$, in case  $w^{Ch,2}_{M}$ and  $w^{Ch,2}_{Ph}$ satisfy conditions 1.-4. and condition  
\begin{equation*}
U_{LDCh}(G) \leqslant U_{LDCh}(I) \Leftrightarrow
26+w^{Ch,2}_{Ph} \cdot 3 \leqslant 27 + w^{Ch,2}_{M} \cdot 3 \Leftrightarrow w^{Ch,2}_{Ph} \cdot 3 \leqslant 1 + w^{Ch,2}_{M}\cdot 3.
\end{equation*}
Therefore, the preference parameters compatible with this piece of preference information are represented by the points in the trapezoid $W_{I\succsim G}$ having vertices $P_1\equiv(\frac{1}{2},\frac{1}{2})$, $P_2\equiv(\frac{1}{6},\frac{1}{2})$, $P_4\equiv(0,\frac{1}{3})$ and $O\equiv(0,0)$ in Figure \ref{FigureLD}. The centroids of $W_{G\succsim I}$ and $W_{I\succsim G}$ are the pairs of weights $(w^{Ch,2}_{M},w^{Ch,2}_{Ph})$ $B_{G\succsim I}\equiv(\frac{1}{18},\frac{4}{9})$,  and $B_{I\succsim G}\equiv(\frac{1}{6},\frac{1}{3})$, respectively, that have been used as examples of parameters for which $U_{LDCh}(G) > U_{LDCh}(I)$ and $U_{LDCh}(G) < U_{LDCh}(I)$.  

The above considerations show that a special attention has to be given to robustness concerns related to the application of the decision model. In general, it is always advisable to avoid to consider only one instance of the decision model compatible with the preference information provided by the DM (in our example the dean) because its choice is arbitrary to some extent. One possible answer to this type of situations is to distinguish preferences that hold for all compatible decision models and preferences that hold for at least one compatible decision model, called respectively necessary and possible preferences (\cite{CGKSml,CGKSen,greco2008ordinal}). With respect to students $G$, $H$ and $I$ we have to conclude that $G$ is necessarily preferred to $H$ as well as $I$ is necessarily preferred to $G$ in consequence of its dominance. Instead, with respect to students $G$ and $I$ we have only a possible preference both for $G$ over $I$ and for $I$ over $G$. \\
One can try to give a more precise evaluation of the range of compatible decision models for which student $G$ is preferred to student $I$ and vice versa by estimating the probability that picking randomly a vector of compatible parameters one gets $U_{LDCh}(G)>U_{LDCh}(I)$ or the opposite. Let us observe that the overall evaluations of students $G, H$ and $I$ involve only the weights $w^{Ch,2}_{M}$ and $w^{Ch,2}_{Ph}$. Consequently, to compare students $G$, $H$ and $I$, the set of the compatible decision models can be identified with the set $W$ of pairs of weights $\left(w^{Ch,2}_{M},w^{Ch,2}_{Ph}\right)$ satisfying conditions 1.-4. (in fact, only conditions 1. and 2., because conditions 3. and 4. are related to the weights $w^{Ch,1}_{M}$ and $w^{Ch,1}_{Ph}$). The set $W$ of compatible weights $\left(w^{Ch,2}_{M},w^{Ch,2}_{Ph}\right)$ are represented by the points in triangle having vertices $O$, $P_1$ and $P_3$ in Figure \ref{FigureLD}. We have already seen that the triangle $W_{G\succsim I}$, which vertices are $P_{2}$, $P_{3}$ and $P_{4}$, is composed of the set of pairs of weights $(w^{Ch,2}_{M},w^{Ch,2}_{Ph})$ for which $G$ is at least as good as $I$, while the trapezoid $W_{I\succsim G}$, which vertices are $O$, $P_{1}$, $P_{2}$ and $P_{4}$, is composed of the set of pairs of weights  for which $I$ is at least as good as $G$. One can suppose a uniform distribution on the set of compatible pairs of weights $W$. Under this hypothesis, the probability of random picking a decision model for which $G$ is preferred to $I$ is given by the ratio between the area of $W_{G\succsim I}$ and the area of $W$, while the analogous probability for $I$ being preferred to $G$ is given by the ratio between the area of $W_{I\succsim G}$ and the area of $W$. Since the area of $W$, $W_{G\succsim I}$ and $W_{I\succsim G}$ are $\frac{1}{8}$, $\frac{1}{72}$ and $\frac{8}{72}$, respectively, we have that the probability of obtaining $G$ preferred over $I$ is $\frac{1}{9}$, while the probability of having $I$ preferred over $G$ is $\frac{8}{9}$. Instead, we have already seen that for all compatible decision models $G$ is preferred to $H$ as well as $I$ is preferred to $H$, so that the corresponding probabilities are equal to 1. One could also look for the probability for which each student $G$, $H$ and $I$ attains the first, the second and the third ranking position. Since we have already seen both $G$ and $I$ are always preferred to $H$, we have to conclude that $H$ takes always the third ranking position, so that his probability of being the third is 1 while the probability of being the first or the second is null. Instead $G$ takes the first ranking position if he is preferred to $I$, otherwise he is the second. Consequently, the probability that $G$ attains the first ranking position is $\frac{1}{9}$, while the probability of being the second is $\frac{8}{9}$. Analogously, we have that for $I$, the probabilities of being the first and the second are $\frac{8}{9}$ and $\frac{1}{9}$, respectively. For $G$ and $I$ the probability of being the third is null. Let us conclude this section by discussing an important element of the decision aiding procedure we presented in this example. It is strongly based on the reference points that are used to create the partition of the interval of evaluations. For instance, in our example, the dean is considering not so good notes, those not greater than 25 and good notes from 26 upwards. Of course these reference points have to be fixed together with the DM because they should express real aspects of his perception of evaluations on considered criteria. Overall, since a decision aiding procedure aims at constructing some arguments explaining the decision to be taken \cite{Roy2010a}, a specific effort has to be put in avoiding to reference points that could permit to represent the preference expressed by the DM but that could be perceived by him as artificial and not credible. Instead, since the idea of reference point is so natural in the human perceptions and decisions (see e.g. \cite{kahneman1992reference,rosch1975cognitive}) it is of fundamental importance for the successful application of the decision aiding procedure to discuss the reference points with the DM that has to accept them as meaningful and sensible. 

In the following we shall show how to apply systematically the decision analysis approach proposed in this introductory example.
   
%%%%%%%%%%%%%%%%%%%%%%%%%%%%%%%%%%%%%%%%%
\section{Basic concepts}\label{BasConc}%%
%%%%%%%%%%%%%%%%%%%%%%%%%%%%%%%%%%%%%%%%%
In this section we shall briefly recall the Choquet integral (Section \ref{ChInt}) and the level dependent Choquet integral (Section \ref{LevDepCh}). A more detailed description of the two preference models can be found in \cite{Grabisch2008} and \cite{greco2011choquet}, respectively.
%%%%%%%%%%%%%%%%%%%%%%%%%%%%%%%%%%%%%%%%%%%%%%%%%%%%%%%%%%%%%%%%%
\subsection{The Choquet integral preference model}\label{ChInt}%%
%%%%%%%%%%%%%%%%%%%%%%%%%%%%%%%%%%%%%%%%%%%%%%%%%%%%%%%%%%%%%%%%%
Let us consider a set of alternatives $A$ to be evaluated with respect to criteria $G=\{g_1,\ldots,g_n\}$ where, for each $g_i\in G$ and for each $a\in A$, $g_i(a)\in\rea$ represents the evaluation of $a$ on $g_{i}$. For the sake of simplicity and without loss of generality, we suppose that all criteria have an increasing direction of preference, that is, the greater $g_i(a)$, the better is $a$ on $g_{i}$. In the following, we shall denote criterion $g_i$ by its index $i$ only. Therefore, the statements $g_i\in G$ or $i\in G$ will be equivalent. \\
A capacity is a set function $\mu:2^{G}\rightarrow[0,1]$ such that $\mu(B)\leqslant\mu(E)$ for all $B\subseteq E\subseteq G$ (monotonicity constraints), $\mu(\emptyset)=0$ and $\mu(G)=1$ (normalization constraints). The M\"{o}bius transform of $\mu$, instead, is a set function $m:2^{G}\rightarrow\rea$ such that, for all $E\subseteq G$, $\mu(E)=\sum_{B\subseteq E}m(B)$ \cite{Shafer}. Consequently, the monotonicity and normalization constraints defined above can be rewritten in terms of $m$ as follows: 
\begin{itemize}
\item $\forall i\in G$ and for all $E\subseteq G\setminus\{i\}$, $\displaystyle\sum_{R\subseteq E}m\left(R\cup\{i\}\right)\geqslant 0$ (monotonicity constraints),
\item $m\left(\emptyset\right)=0$, $\displaystyle\sum_{E\subseteq G}m(E)=1$ (normalization constraints).
\end{itemize} 
In this case, the Choquet integral of $\mathbf{x}=\left(x_1,\ldots,x_n\right)\in\rea^{n}$ w.r.t. $\mu$ can be written as 

\begin{equation}\label{Choquet}
Ch(\mathbf{x},\mu)=\sum_{i=1}^{n}\mu(N_i)\left[x_{(i)}-x_{(i-1)}\right]
\end{equation}

\noindent where $(\cdot)$ stands for a permutation of the indices of criteria such that $0=x_{(0)}\leqslant\ldots\leqslant x_{(i)}\leq\ldots\leqslant x_{(n)}$ and $N_{i}=\{j\in G: x_j\geqslant x_{(i)}\}$ for all $i=1,\ldots,n.$ From now on, we shall use $a$ to denote an alternative belonging to $A$ and $\mathbf{a}$ to indicate the vector which components are the evaluations of $a$ on the considered criteria.

The use of (\ref{Choquet}) involves the knowledge of $2^n-2$ parameters (one for each subset of $G$ apart from $\emptyset$ and $G$ since $\mu(\emptyset)=0$ and $\mu(G)=1$). Anyway, for real world applications, $2$-additive capacities are enough to represent the preferences of the DM where a capacity is said $k$-additive iff its M\"{o}bius transform $m$ is such that $m(E)=0$ for all $E\subseteq G: |E|>k$ \cite{grabisch1997k}. In this case, the Choquet integral of $\mathbf{x}$ can be rewritten as  

$$
Ch(\mathbf{x},\mu)=\sum_{i=1}^{n}m\left(\{i\}\right)x_i+\sum_{\{i,j\}\subseteq G}m\left(\{i,j\}\right)\min\{x_i,x_j\}
$$

\noindent while monotonicity and normalization constraints become 

\begin{itemize}
\item $\displaystyle m\left( \emptyset\right) =0,\,\,\,\sum_{i\in G}m\left(\left\{ i\right\}\right)+\sum_{\left\{ i,j\right\}\subseteq G}m\left(\left\{ i,j\right\}\right) =1,$ (normalization constraints)
\item $
\left\{
\begin{array}{l}
\displaystyle m\left( \left\{ i\right\}\right)\geq 0,\; \forall i \in G,\\[2mm]
\displaystyle m\left( \left\{ i\right\}\right) +{\sum_{j\in E}}\:m\left(\left\{ i,j\right\}\right) \geqslant 0, \; \forall i\in G
\:\mbox{and}\:\forall\: E\subseteq G \setminus \left\{i\right\},$ $E\neq\emptyset
\end{array}
\right.
$ (monotonicity constraints).
\end{itemize}

%%%%%%%%%%%%%%%%%%%%%%%%%%%%%%%%%%%%%%%%%%%%%%%%%%%%%%%%%%%%%%%%%%%%
\subsection{The level dependent Choquet integral}\label{LevDepCh}%%%
%%%%%%%%%%%%%%%%%%%%%%%%%%%%%%%%%%%%%%%%%%%%%%%%%%%%%%%%%%%%%%%%%%%%
A generalized capacity $\mu^{L}$ is a function $\mu^{L}:2^{G}\times [\alpha,\beta]\rightarrow [0,1]$, with $[\alpha,\beta]\subseteq\mathbb{R}$, such that:
\begin{description}
\label{ldcconditions}
 \item \textbf{1.a}) for all $t\in\left[\alpha,\beta\right]$ and for all $B \subseteq E \subseteq G, \; \mu^{L}(B,t)\leqslant \mu^{L}(E,t)$,
 \item \textbf{2.a}) for all $t\in\left[\alpha,\beta\right]$, $\mu^{L}(\emptyset,t)=0$ and $\mu^{L}(G,t)=1$,
 \item \textbf{3.a}) for all $t\in\left[\alpha,\beta\right]$ and for all $E\subseteq G,\; \mu^{L}(E,t)$ is Lebesgue measurable if considered as a function of $t$.
\end{description}

\noindent Let us observe that items \textbf{1.a}) and \textbf{2.a}) above imply that $\mu^{L}$ gives a capacity for each $t\in[\alpha,\beta]$. 

The level dependent Choquet integral of $\mathbf{x}=[x_{1},\ldots, x_{n}]\in [\alpha,\beta]^{n}$ w.r.t. $\mu^{L}$, can be defined as follows:

$${Ch}^{L}(\mathbf{x}, \mu^{L})=\displaystyle\int^{\max \{x_{1},\ldots,x_{n}\}}_{\min\{x_{1},\ldots,x_{n} \}}\mu^{L}(\{i \in G : x_{i}\geqslant t \},t)\;dt + \min \{x_{1},\ldots,x_{n}\}.$$

% \noindent where $E(\mathbf{x},t)=\{i \in G : x_{i}\geq t \}$.

\noindent If the values of criteria are normalized between 0 and 1 and, consequently, $\mathbf{x}\in [0,1]^{n}$, we get
 
$$
{Ch}^{L}(\mathbf{x}, \mu^{L})=\displaystyle\int^{1}_{0}\mu^{L}(\{i \in G : x_{i}\geqslant t \},t)\;dt.
$$

%%%%%%%%%%%%%%%%%%%%%%%%%%%%%%%%%%%%%%%%%%%%%%%%%%%%%%%%%%%%%%%%%%%%%%%
\subsubsection{Interval level dependent capacity}\label{IntLevDepCh}%%%
%%%%%%%%%%%%%%%%%%%%%%%%%%%%%%%%%%%%%%%%%%%%%%%%%%%%%%%%%%%%%%%%%%%%%%%
The use of the level dependent Choquet integral can be limited by noticing that it involves the definition of an arbitrary number of capacities $\mu^{L}(\cdot,t)$, one for each level $t \in [\alpha,\beta].$ To overcome this problem, in \cite{greco2011choquet} the authors suggest the use of an \textit{interval level dependent capacity}. A generalized capacity $\mu^{L}$ is called interval level dependent if there exists 
\begin{itemize}
\item a partition $a_{0},a_{1},\ldots, a_{p-1},a_{p}$, of the interval $[\alpha,\beta]$, where  $\alpha=a_{0}<a_{1}<\ldots<a_{p-1}<a_{p}=\beta$, 
\item $p$ capacities $\mu_{1},\ldots,\mu_{p}$ defined on $2^{G}$,
\end{itemize}
\noindent such that, for all $E\subseteq G$, $\mu^{L}(E,t)=\mu_{r}(E)$  if $t \in [a_{r-1},a_{r}[$, $r=1,\ldots,p-1,$ and $\mu^{L}(E,t)=\mu_{p}(E)$ if $t\in[a_{p-1},a_p]$. For the sake of simplicity, from now on, even if we shall write $[a_{r-1},a_{r}[$ with $r=1,\ldots,p$, we shall consider as last interval the closed interval $[a_{p-1},a_{p}]$. Moreover, let us observe that in defining the partition we can also consider closed intervals $\left([a_{r-1},a_r]\right)$, half-opened intervals ($[a_{r-1},a_r[$ or $]a_{r-1},a_r]$) or opened intervals $\left(]a_{r-1},a_r[\right)$. What is important is that all subintervals of the decomposition are pairwise disjoint. 

\noindent If $\mu^{L}$ is an interval level dependent capacity, the following theorem allows to compute the level dependent Choquet integral of $\mathbf{x}$ w.r.t. $\mu^{L}$ as the sum of $p$ Choquet integrals, one for each interval $[a_{r-1},a_{r}[,\; r=1,\ldots,p$:
\begin{theorem}(see \cite{greco2011choquet})
\label{theorem1}
If $\mu^{L}$ is an interval level dependent capacity on $2^G$ relative to the breakpoints\\
$a_{0},a_{1},\ldots, a_{p-1},a_{p}\in [\alpha,\beta]$ and to the capacities $\mu_{1},\ldots,\mu_{p}$, then for each $\mathbf{x}\in [\alpha,\beta]^{n}$

$$
{Ch}^{L}(\mathbf{x},\mu^{L})=\displaystyle \sum_{r=1}^{p}Ch({\mathbf{x}^{r},\mu_{r}}),
$$
\noindent where $\mathbf{x}^{r}\in [0,a_{r}-a_{r-1}]^{n}$ is the vector which components $x_{i}^{r}$ are defined as

$$
x_{i}^{r} = \left\{ 
\begin{array}{ll}
         0             & \mbox{if $ \quad x_{i} < a_{r-1}$}\\   
			   x_{i}-a_{r-1} & \mbox{if $ \quad a_{r-1} \leqslant x_{i} < a_{r}$},\\
				 a_{r}-a_{r-1} & \mbox{if $ \quad x_{i} \geqslant a_{r}$}\\
         \end{array} 
				\right. 
$$ 
\noindent for all $i=1,\ldots,n.$
\end{theorem}
\noindent Consequently, the level dependent Choquet integral can be written as

\begin{equation}\label{ILDCh}
{Ch}^{L}(\mathbf{x},\mu^{L})=\displaystyle \sum_{r=1}^{p}Ch({\mathbf{x}^{r},\mu_{r}})=\displaystyle\sum_{r=1}^{p}\sum_{i=1}^{n}\mu_{r}(N^{r}_{i})\left[{x}^{r}_{(i)}-{x}^{r}_{(i-1)}\right],
\end{equation}

\noindent where ${(\cdot)}$ stands for a permutation of the indices of criteria so that $0=x^{r}_{(0)}\leqslant\ldots\leqslant x^{r}_{(i)}\leq\ldots\leqslant x^{r}_{(n)}$ and $N_{i}^{r}=\{j\in G: x^{r}_j\geqslant x^{r}_{(i)}\}$. 

It is interesting to observe that a result similar to the one shown in Theorem \ref{theorem1} was proposed in \cite{mesiar2015choquet} where the Choquet integral with respect to an interval level dependent capacity was formulated in terms of an ordinal sum of Choquet integrals based on capacities $\mu_r$. More precisely, on the basis of \cite{mesiar2000new}, consider 
\begin{itemize}
	\item the above partition of $\left[\alpha,\beta\right]$ obtained with $\alpha=a_{0}<a_{1}<\ldots<a_{p-1}<a_{p}=\beta$, 
\item $p$ aggregation functions $A_r, r=1,\ldots,p$, that are, functions $A_r:[a_{r-1},a_r]^n\rightarrow [a_{r-1},a_r]$  non decreasing in their arguments and such that $A_r(a_{r-1},\ldots,a_{r-1})=a_{r-1}$ and $A_r(a_{r},\ldots,a_{r})=a_{r}$, 
	\item an increasing bijection $\varphi:[\alpha,\beta]\rightarrow\mathbb{R}.$  
\end{itemize}
Then the $\varphi$-ordinal sum of aggregation functions $A_1,\ldots,A_p$ is given by 
\begin{equation}\label{Ordinal_sum}
\mathbf{A}\left(\varphi,A_1,\ldots,A_p;x_1,\ldots,x_n\right)=\varphi^{-1}\left(\displaystyle \sum_{r=1}^{p}\left[\varphi\left(A_r\left(\tilde{x}^r_1,\ldots,\tilde{x}_n^r\right)\right)-\varphi\left(a_{r-1}\right)\right]\right),
\end{equation}
with $\tilde{x}^r_i=max\left(a_r,min\left(a_{r-1},x_i\right)\right), i=1,\ldots,n$.  

\cite{mesiar2015choquet} proves that  
\begin{equation}\label{Ordinal_sum_ldc}
{Ch}^{L}(\mathbf{x},\mu^{L})=\mathbf{A}\left(id,Ch^1({\mathbf{x},\mu_{1}}),\ldots,Ch^p({\mathbf{x},\mu_{p})};x_1,\ldots,x_n\right)
\end{equation}
with $id(x)=x$ for all $x \in [\alpha,\beta]$ and $Ch^r({\mathbf{x},\mu_{r}})$ being the restriction of $Ch({\mathbf{x},\mu_{r}})$ to $[a_{r-1},a_r]^n$, that is,
\begin{equation}\label{Ordinal_sum_ldc_}
{Ch}^{L}(\mathbf{x},\mu^{L})=\displaystyle \sum_{r=1}^{p}\left(Ch\left(\mathbf{\tilde{x}^r},\mu_{r}\right)-a_{r-1}\right).
\end{equation}
The formulation (\ref{Ordinal_sum_ldc}) has the merit of expressing the level dependent Choquet integral based on an interval level dependent capacity in terms of ordinal sum which is an important issue originally proposed by Birkhoff \cite{birkhoff1940lattice} for posets and lattices and after applied for $t$-norms and  $t$-conorms \cite{klement2000triangular} and copulas \cite{nelsen2007introduction}. From a more operational point of view, formulations (\ref{ILDCh}) and (\ref{Ordinal_sum_ldc_}) permit to represent the level dependent Choquet integral based on an interval level dependent capacity in terms of usual Choquet integrals based on capacities $\mu_r, r=1,\ldots,p$. In fact the addends of the sum in (\ref{ILDCh}) and  (\ref{Ordinal_sum_ldc_}) are the same, that is
\begin{equation}\label{GMG_MS}
Ch({\mathbf{x}^{r},\mu_{r}})=Ch({\mathbf{\tilde{x}^r},\mu_{r}})-a_{r-1}, r=1,\ldots,n,
\end{equation}
because 
\begin{equation}\label{ETR}
\tilde{x}^r_i=x_i^r+a_{r-1}, r=1,\ldots,p, i=1,\ldots,n,
\end{equation}
and in consequence of the translation invariance of the Choquet integral, that is, 
\begin{equation}\label{TI}
Ch({\mathbf{x+c},\mu})=Ch({\mathbf{x},\mu})+c
\end{equation}
for all capacities $\mu$, all vectors $\mathbf{x} \in \mathbb{R}^n$ and all constants $c \in \mathbb{R}$ with $\mathbf{c}=[c, \ldots, c]\in \mathbb{R}^n$. Indeed, by (\ref{ETR}) and (\ref{TI}) one can get  
\begin{equation}\label{GMG_MS_}
Ch({\mathbf{\tilde{x}^r},\mu_{r}})=Ch({\mathbf{x}^{r}+\mathbf{a_{r-1}},\mu_{r}})=Ch({\mathbf{x}^{r},\mu_{r}})+a_{r-1}, r=1,\ldots,n,
\end{equation}
and, consequently, (\ref{GMG_MS}).
%
 %an ordinal sums of
%\end{itemize}
%an increasing bijection $\phi:[\alpha,\beta]\rightarrow\mathbb{R}$,  
%
%Proposition 7 Let 0 = a0 < a1 < ··· < an = 1, and let Ai : [ai−1, ai]
%m →
%[ai−1, ai], i = 1, . . . , n, be continuous aggregation functions. Let f : [0, 1] → R
%be a continuous strictly monotone non-bijective function. Then the mapping
%A : [0, 1]m → [0, 1] given by
%A(x1,...,xm) = f −1
%
%n
%i=1
%f
%
%Ai(x(i)
%1 ,...,x(i)
%m )
%
%−
%n
%−1
%i=1
%f(ai)
%
%,
%where x(i) = Max (ai−1, Min(ai, x)), is a continuous aggregation function.
%
%
%given an automorphism  a considered in 
%
%

%
%Ordinal sums are well known for t-norms, copulas, semicopulas (the same formula based
%on Min), as well as for t-conorms (a dual formula based on Max). In order to unify
%all these formulae in a unique general formula, in [6], the concept of ϕ-ordinal sums of
%aggregation functions was introduced.

%Because of the manageability of the previous results, from now on, we will refer to the interval level dependent capacity. 

%%%%%%%%%%%%%%%%%%%%%%%%%%%%%%%%%%%%%%%%%%%%%%%%%%%%%%%%%%%%%%%%%%%%%%%%%%%%%%%%%
\subsubsection{M\"{o}bius transform of level dependent capacity}\label{MobRep}%%%
%%%%%%%%%%%%%%%%%%%%%%%%%%%%%%%%%%%%%%%%%%%%%%%%%%%%%%%%%%%%%%%%%%%%%%%%%%%%%%%%%
Given a level dependent capacity $\mu^{L}$, its M\"{o}bius transform is a function $m^{L}:2^{G}\times(\alpha, \beta)\rightarrow \mathbb{R}$, such that, for all $E\subseteq G$ and for all $t \in [\alpha,\beta]$ 

\begin{equation*}
\mu^{L}(E,t)=\displaystyle\sum_{B\subseteq E}m^{L}(B,t).
\end{equation*}

\noindent In terms of M\"{o}bius transform, properties \textbf{1.a}) and \textbf{2.a}) are restated  as:

\begin{description}
  \item \textbf{1.b}) $\displaystyle m^{L}(\emptyset,t)=0,\,\,\,\sum_{E\subseteq G}m^{L}(E,t)=1,\;\forall t\in [\alpha, \beta]$,
  \item \textbf{2.b}) $\displaystyle \forall \: i \in G\;\mbox{and}\;\forall E\subseteq G\setminus\left\{i\right\},\; \sum_{R\subseteq E}\:m^{L}(R\cup\left\{i\right\},t)\geqslant 0,\;\forall t\in [\alpha,\beta]$.
\end{description}

Considering an interval level dependent capacity where $a_0,a_1,\ldots,a_p$ are the elements decomposing the interval $[\alpha,\beta]$ and $m_r$ the M\"{o}bius transform of the capacities $\mu_r$, $r=1,\ldots,p,$ the level dependent Choquet integral (\ref{ILDCh}) can be written in the following way

\begin{equation*}
{Ch}^{L}(\mathbf{x},\mu^{L})=\displaystyle \sum_{r=1}^{p}Ch({\mathbf{x}^{r},\mu_{r}})=\displaystyle\sum_{r=1}^{p}\left[\sum_{E\subseteq G}m_r(E)\;\displaystyle\min_{i\in E} x_{i}^{r}\right]
\end{equation*}

\noindent where the $x^{r}_{i},\; i=1,\ldots,n$ have been defined above.

%%%%%%%%%%%%%%%%%%%%%%%%%%%%%%%%%%%%%%%%%%%%%%%%%%%%%%%%%%%%%%%%%%%%%%%%%%%%%%%%%%%%%%
\subsubsection{$k$-additive interval level dependent capacity}\label{KOrdLevDepCap}%%%
%%%%%%%%%%%%%%%%%%%%%%%%%%%%%%%%%%%%%%%%%%%%%%%%%%%%%%%%%%%%%%%%%%%%%%%%%%%%%%%%%%%%%%
To apply in a more manageable way the level dependent Choquet integral, similarly to what has been done for a capacity, $k$-order additive interval level dependent capacities can be defined. An interval level dependent capacity $\mu^{L}$ is said $k$-additive $(1\leqslant k\leqslant n-1)$ if, for all $t \in [\alpha, \beta]$, $m^{L}(E,t)=0$ for all $E \subseteq G$ such that $|E|>k$. Since, as observed above, giving a weight to each criterion and to each pair of criteria is enough to represent the preferences of the DM, 2-additive interval level dependent capacities are used in practice. The value that a $2$-additive interval level dependent capacity $\mu^{L}$ assigns to a set $E\subseteq G$ can be expressed in terms of its M\"{o}bius transform as follows: $\forall E\subseteq G\;\mbox{and}\;\forall t\in [\alpha,\beta]$

\begin{equation*}
\mu^{L}(E,t)=\underset{i\in E}{\sum }m^{L}\left( \left\{ i\right\},t \right) +\underset{\left\{ i,j\right\}
\subseteq E}{\sum }m^{L}\left( \left\{ i,j\right\},t \right).
\end{equation*}

With regard to 2-additive interval level dependent capacities, properties \textbf{1.b)} and \textbf{2.b)} can be restated as follows:

\begin{description}
\item[1.c)] $\displaystyle m^{L} \left( \emptyset,t \right) =0,\,\,\,\sum_{i\in G}m^{L}\left(\left\{ i\right\},t \right) +\sum_{\left\{ i,j\right\}\subseteq G}m^{L}\left(\left\{ i,j\right\},t \right) =1,$
\item[2.c)] $
\left\{
\begin{array}{l}
\displaystyle m^{L}\left( \left\{ i\right\},t \right) \geqslant 0,\; \forall i \in G,\\[2mm]
\displaystyle m^{L}\left( \left\{ i\right\},t \right) +{\sum_{j\in E}}\:m^{L}\left( \left\{ i,j\right\},t \right) \geqslant 0, \; \forall i\in G
\:\mbox{and}\:\forall\: E\subseteq G \setminus \left\{i\right\},$ $E\neq\emptyset.
\end{array}
\right.
$
\end{description}

\begin{prop}\label{kAdditM}
Given
\begin{itemize}
\item an interval level dependent capacity $\mu^{L}$ over $2^{G}\times [\alpha,\beta]$, with $a_0,\ldots,a_p$ being a partition of the interval $[\alpha,\beta]$ and $\mu_r$ the corresponding capacities, $r=1,\ldots,p,$
\item $m^{L}$ and $m_{r}$ the M\"{o}bius transforms of $\mu^{L}$ and $\mu_{r}$, respectively, 
\end{itemize}
then $\mu^{L}$ is a $k$-additive interval level dependent capacity if and only if all $\mu_r$ are $k$-additive, $r=1,\ldots,p$.
\end{prop}
\begin{proof}
It is a straightforward consequence of the definitions of $k$-additive capacities and $k$-additive interval level dependent capacities.
\end{proof}

Considering Proposition \ref{kAdditM}, if $\mu^{L}$ is $2$-additive, then the level dependent Choquet integral may be also reformulated as follows:

\begin{equation*}
{Ch}^{L}(\mathbf{x},\mu^{L})=\displaystyle\sum_{r=1}^{p}Ch(\mathbf{x}^{r},\mu_r)=\displaystyle\sum_{r=1}^{p}\left[\sum_{i\in G}m_r(\{i \})x^{r}_{i} + \sum_{\{i,j\}\subseteq G}m_r(\{i,j\})\displaystyle\min \{x^{r}_{i},x^{r}_{j} \}\right]
\end{equation*}

\noindent where the $x_i^{r}$ , $i = 1, . . . , n$, have been defined above.

%%%%%%%%%%%%%%%%%%%%%%%%%%%%%%%%%%%%%%%%%%%%%%%%%%%%%%%
\subsubsection{Importance indices}\label{ImpIndices}%%%
%%%%%%%%%%%%%%%%%%%%%%%%%%%%%%%%%%%%%%%%%%%%%%%%%%%%%%%
Taking into account a capacity, the importance of a criterion $i$ is not dependent on itself only but also on its contribution to all possible coalitions of criteria. For this reason, the Shapley index, defining the importance of a criterion $i$, has been introduced in \cite{shapley}. In case of a level dependent capacity, the importance of a criterion $i$ may (but has not to) vary when the level $t$ is changed. Consequently, given a generalized capacity $\mu^{L}$, we have that \cite{greco2011choquet}:

\begin{equation}\label{ShapleyGeneralizedCapacity}
\phi(\mu^{L},i,t)=\displaystyle \sum_{E\subseteq G\setminus\{i\}}\frac{|E|!(|G\setminus E|-1)!}{|G|!}[\mu^{L}(E\cup \{i \},t)-\mu^{L}(E,t)].
\end{equation}
Considering the M\"{o}bius transform $m^{L}$ of $\mu^{L}$, the importance index $\phi(\mu^{L},i,t)$ can be rewritten as
\begin{equation}\label{ShapleyGeneralizedMobius}
\phi(\mu^{L},i,t)=\displaystyle \sum_{E\subseteq G\setminus \{i\}} \frac{m^{L}(E\cup \{i \},t)}{|E\cup\{i\}|}.
\end{equation}
Finally, if $\mu^{L}$ is an interval level dependent capacity, then 
\begin{equation}\label{ShapleyIntLD}
\phi(\mu^{L},i,t)=\sum_{E\subseteq G\setminus\{i\}}\frac{m_r(E\cup\{i\})}{|E\cup\{i\}|},\;\;\forall t\in[a_{r-1},a_r[.
\end{equation}

From (\ref{ShapleyGeneralizedCapacity})-(\ref{ShapleyIntLD}), the comprehensive importance of criterion $i$ can be obtained as \cite{greco2011choquet}:

$$
\phi(\mu^{L},i)=\frac{\int_{\alpha}^{\beta}\phi(\mu^{L},i,t)\;dt}{\beta-\alpha}.
$$

In particular, for an interval level dependent capacity, the importance of criterion $i$ on the interval $[a_{r-1},a_r[$ is given by 

$$
\phi(\mu^{L},i,[a_{r-1},a_r[)=\displaystyle\sum_{E\subseteq G\setminus\{i\}}\frac{m_r(E\cup\{i\})}{|E\cup\{i\}|}
$$

\noindent and, consequently,

$$
\phi(\mu^{L},i)=\displaystyle\sum_{r=1}^{p}\left\{\sum_{E\subseteq G\setminus\{i\}}\frac{m_r(E\cup\{i\})}{|E\cup\{i\}|}\cdot\frac{a_r-a_{r-1}}{\beta-\alpha}\right\}.
$$

\noindent For a $2$-additive interval level dependent capacity, the previous index can be rewritten as 

\begin{equation*}
\phi(\mu^{L},i)=\displaystyle\sum_{r=1}^{p}\left\{\left[m_r(\{i\})+\displaystyle\sum_{j\in G\setminus\{i\}}\frac{m_r(\{i,j\})}{2}\right]\cdot\frac{a_r-a_{r-1}}{\beta-\alpha}\right\}.
\end{equation*}

%%%%%%%%%%%%%%%%%%%%%%%%%%%%%%%%%%%%%%%%%%%%%%%%%%%%
\subsubsection{Interaction indices}\label{IntInd}%%%
%%%%%%%%%%%%%%%%%%%%%%%%%%%%%%%%%%%%%%%%%%%%%%%%%%%%
With respect to the Choquet integral, interaction indices have been introduced by Murofushi and Soneda \cite{Murofushi1993} considering only couples of criteria, and by Grabisch \cite{grabisch1997k} with respect to all possible subsets of criteria. \\

\noindent Given a level dependent capacity $\mu^{L},\; E\subseteq G$ and $t \in [\alpha, \beta]$, the interaction index can be computed as \cite{greco2011choquet}

\begin{equation*}
I(\mu^{L},E,t)=\displaystyle \sum_{B\subseteq G\setminus E} \frac{|B|!(|G|-|B|-|E|)!}{|G|!}\left(\displaystyle \sum_{C\subseteq E}(-1)^{\left|E\setminus C\right|}\mu^{L}(C\cup B,t)\right).
\end{equation*}
The interaction index $I(\mu^{L},E,t)$ can be represented in terms of M\"{o}bius transform as follows:

\begin{equation*}
I(\mu^{L},E,t)=\displaystyle \sum_{B\subseteq G\setminus E} \frac{m^{L}(E\cup B,t)}{|B|+1}.
\end{equation*}

\noindent If $\mu^{L}$ is an interval level dependent capacity, then 

\begin{equation*}
I(\mu^{L},E,t)=\displaystyle \sum_{B\subseteq G\setminus E}\frac{m_r(E\cup B)}{|B|+1},\;\;\forall t\in[a_{r-1},a_r[.
\end{equation*}

\noindent As a consequence, the importance of criteria in $E$ for the interval $[a_{r-1},a_r[$ is 

$$
I(\mu^{L},E,[a_{r-1},a_r[)=\displaystyle\sum_{B\subseteq G\setminus E}\frac{m_r(E\cup B)}{|B|+1}
$$

\noindent and, therefore, 

\begin{equation}\label{IntIndIntLD}
I(\mu^{L},E)=\displaystyle\sum_{r=1}^{p}\left\{\sum_{B\subseteq G\setminus E}\frac{m_r(E\cup B)}{|B|+1}\cdot\frac{a_{r}-a_{r-1}}{\beta-\alpha}\right\}.
\end{equation}

If $E=\{i,j\}$ and $\mu^{L}$ is a 2-additive interval level dependent capacity, then (\ref{IntIndIntLD}) can be rewritten as 

\begin{equation*}
I(\mu^{L},\{i,j\})=\sum_{r=1}^{p}\left\{m_r(\{i,j\})\cdot\frac{a_r-a_{r-1}}{\beta-\alpha}\right\}.
\end{equation*}

%%%%%%%%%%%%%%%%%%%%%%%%%%%%%%%%%%%%%%%%%%%%%%%%%%%%%%%%%%%%%%%%%%%%%%%%%%%%%%%%%%%%%%%%%%%%%%%%%%%%%%%%%%%%%%%
\section{NAROR and SMAA applied to the level dependent Choquet integral preference model}\label{NARSMLevDep}%%%
%%%%%%%%%%%%%%%%%%%%%%%%%%%%%%%%%%%%%%%%%%%%%%%%%%%%%%%%%%%%%%%%%%%%%%%%%%%%%%%%%%%%%%%%%%%%%%%%%%%%%%%%%%%%%%%
In this section we shall describe the extension of NAROR and SMAA to the level dependent Choquet integral. 
%%%%%%%%%%%%%%%%%%%%%%%%%%%%%%%%%%%%%%%%%%%%%%%%%%%%%%%%%%%%%%%%%%%%%%%%%%%%%%%%
\subsection{Non Additive Robust Ordinal Regression (NAROR)}\label{NAROR_des}%%%%
%%%%%%%%%%%%%%%%%%%%%%%%%%%%%%%%%%%%%%%%%%%%%%%%%%%%%%%%%%%%%%%%%%%%%%%%%%%%%%%%
Every Multiple Criteria Decision Aiding (MCDA) method needs some parameters to be implemented \cite{GreFigEhr}. In particular, when the parameters related to the DM's preferences have to be set, it can be chosen a direct or an indirect technique to get them. In the first case, the DM provides all the values of the parameters requested by the method. In the second case, the DM provides only some statements which are subsequently used to obtain the parameters requested by the model.

The NAROR \cite{angilella2010non} belongs to the family of ROR methods (see \cite{CGKSml,CGKSen,greco2008ordinal}) and, as any ROR method, it is based on the indirect preference information.\\
Regarding this preference information, the DM can provide (among parenthesis the constraints translating the corresponding piece of preference information):

\begin{itemize}
	\item a possibly incomplete binary relation $\succsim^{DM}$ on a subset of alternatives $A^{*}\subseteq A$ for which $a\succsim^{DM}b$ iff $a$ is at least as good as $b$ $\left(Ch^{L}(\mathbf{a},\mu^L)\geqslant Ch^{L}(\mathbf{b},\mu^L)\right)$, 
	\item a possibly incomplete binary relation $\succsim^{DM}$ on a subset of criteria $G^{*}\subseteq G$, for which $g_i\succsim^{DM}g_j$ iff $g_i$ is at least as important as $g_j$ $\left(\phi(\mu^{L},i)\geqslant \phi(\mu^{L},j)\right)$, 
	\item a possibly incomplete binary relation $\succsim^{DM}_{[a_{r-1},a_{r}[}$ on a subset of criteria $G^{*}\subseteq G$, for which \\ $g_i\succsim_{[a_{r-1},a_{r}[}^{DM}g_j$ iff, considering the evaluations in the interval $[a_{r-1},a_r[$, $g_i$ is at least as important as $g_j$  $\left(\phi(\mu^{L},i,[a_{r-1},a_r[)\geqslant \phi(\mu^{L},j,[a_{r-1},a_r[)\right)$\footnote{On the basis of $\succsim^{DM}$, a strict preference relation $\succ^{DM}$ and an indifference relation $\sim^{DM}$ can be defined such that $a\succ^{DM}b$ iff $a\succsim^{DM}b$ and $not(b\succsim^{DM}a)$, while $a\sim^{DM}b$ iff $a\succsim^{DM}b$ and $b\succsim^{DM}a$. Analogously, on the basis of $\succsim_{[a_{r-1},a_{r}[}^{DM}$, the relations $\succ_{[a_{r-1},a_{r}[}^{DM}$ and $\sim_{[a_{r-1},a_{r}[}^{DM}$ can be defined.} ,
	\item positive or negative interactions between criteria $g_i$ and $g_j$ in a comprehensive way ($I(\mu^{L},\{i,l\})>~0$ if $g_i$ and $g_j$ are positively interacting or $I(\mu^{L},\{i,j\})<0$ if $g_i$ and $g_j$ are negatively interacting) or considering a particular interval of evaluations $\left(I(\mu^{L},\{i,j\},[a_{r-1},a_r[)\right.>0$ if $g_i$ and $g_j$ are positively interacting or $I(\mu^{L},\{i,j\},[a_{r-1},a_r[)<0$ if $g_i$ and $g_j$ are negatively interacting).
\end{itemize}

Considering a 2-additive interval level dependent capacity $\mu^{L}$ relative to the breakpoints $a_0,a_1,\ldots,a_p$ and the M\"{o}bius transform $m_{r}$ of the capacities $\mu_r$, $r=1,\ldots,p$, to check if there exists at least one interval level dependent Choquet integral compatible with the preferences provided by the DM, the following LP problem has to be solved:

\begin{equation}\label{LPExistence}
\begin{array}{l}
\;\;\;\;\varepsilon^{*}=max\;\varepsilon,\;\mbox{subject to},\\[2mm]
\left.
\begin{array}{l}
\;\;Ch^{L}(\mathbf{a},\mu^L)\geqslant Ch^{L}(\mathbf{b},\mu^L)+\varepsilon, \;\mbox{if}\; a\succ^{DM}b,\\[2mm]
\;\;Ch^{L}(\mathbf{a},\mu^L)=Ch^{L}(\mathbf{b},\mu^L), \;\mbox{if}\; a\sim^{DM}b,\\[2mm]
\;\phi(\mu^{L},i)\geqslant\phi(\mu^{L},j)+\varepsilon, \;\mbox{if}\; g_i\succ^{DM}g_j,\\[2mm]
\;\phi(\mu^{L},i)=\phi(\mu^{L},j), \;\mbox{if}\; g_i\sim^{DM}g_j,\\[2mm]
\;\phi(\mu^{L},i,[a_{r-1},a_r[)\geqslant\phi(\mu^{L},j,[a_{r-1},a_r[)+\varepsilon, \;\mbox{if}\; g_i\succ_{[a_{r-1},a_r[}^{DM}g_j,\\[2mm]
\;\phi(\mu^{L},i,[a_{r-1},a_r[)=\phi(\mu^{L},j,[a_{r-1},a_r[), \;\mbox{if}\; g_i\sim^{DM}_{[a_{r-1},a_r[}g_j,\\[2mm]
\; I(\mu^{L},\{i,j\})\geqslant\varepsilon \;\;\mbox{if $g_{i}$ and $g_j$ are comprehensively positively interacting},\\[2mm]
\; I(\mu^{L},\{i,j\},[a_{r-1},a_r[)\geqslant\varepsilon \;\;\mbox{if $g_{i}$ and $g_j$ are positively interacting considering}\\[2mm]
\;\;\;\;\;\;\mbox{evaluations in the interval $[a_{r-1},a_r[$},\\[2mm]
\; I(\mu^{L},\{i,j\})\leq-\varepsilon \;\;\mbox{if $g_{i}$ and $g_j$ are comprehensively negatively interacting},\\[2mm]
\; I(\mu^{L},\{i,j\},[a_{r-1},a_r[)\leq-\varepsilon \;\;\mbox{if $g_{i}$ and $g_j$ are negatively interacting considering}\\[2mm]
\;\;\;\;\;\;\mbox{evaluations in the interval $[a_{r-1},a_r[$},\\[2mm]
\left.
\begin{array}{l}
\displaystyle m_r\left(\emptyset\right)=0,\,\,\,\sum_{i\in G}m_r\left(\left\{ i\right\}\right) +\sum_{\left\{ i,j\right\}\subseteq G}m_r\left(\left\{i,j\right\}\right)=1,\\[2mm]
\displaystyle m_r\left(\left\{i\right\}\right)\geqslant 0,\; \forall i \in G,\\[2mm]
\displaystyle m_r\left(\left\{i\right\}\right)+{\sum_{j\in E}}\:m_r\left(\left\{i,j\right\}\right)\geqslant 0, \; \forall i\in G \:\mbox{and}\:\forall\: E\subseteq G \setminus \left\{i\right\},$ $E\neq\emptyset. \\[2mm]
\end{array}
\right\}\forall r=1,\ldots,p,\\[2mm]
\end{array}\right\}E^{DM}\\
\end{array}
\end{equation}

If $E^{DM}$ is feasible and $\varepsilon^{*}>0$, then there exists at least one interval level dependent capacity compatible with the preferences provided by the DM, otherwise, there is not any interval level dependent capacity compatible and, therefore, the constraints causing this incompatibility can be identified by using one of the methods presented in \cite{mousseau2003resolving}. 

In general, if there exists one interval level dependent capacity compatible with the preferences provided by the DM, there exists more than one and, therefore, the choice of only one of them can be considered arbitrary to some extent. For this reason, ROR takes into account all of them defining a necessary $\succsim^N$ and a possible $\succsim^P$ preference relation. The necessary preference relation holds between two alternatives $a$ and $b$ if $a$ is at least as good as $b$ for all compatible models, while the possible preference relation holds between $a$ and $b$ if $a$ is at least as good as $b$ for at least one compatible model. \\
From a computational point of view, the two preference relations can be computed as follows: 
\begin{itemize}
\item $a\succsim^{N}b$ iff $E^{N}(a,b)$ is infeasible or $\varepsilon^{N}\leq 0$ where $\varepsilon^{N}=\max\varepsilon$ subject to $E^{N}(a,b)$ and 
$$
E^{N}(a,b)=\{Ch^{L}(\mathbf{b},\mu^L)\geqslant Ch^{L}(\mathbf{a},\mu^L)+\varepsilon\}\cup E^{DM},
$$
\item $a\succsim^{P}b$ iff $E^{P}(a,b)$ is feasible and $\varepsilon^{P}>0$ where $\varepsilon^{P}=\max\varepsilon$ subject to $E^{P}(a,b)$ and 
$$
E^{P}(a,b)=\{Ch^{L}(\mathbf{a},\mu^L)\geqslant Ch^{L}(\mathbf{b},\mu^L)\}\cup E^{DM}.
$$
\end{itemize}

%%%%%%%%%%%%%%%%%%%%%%%%%%%%%%%%%%%%%%%%%%%%%%%%%%%%%%%%%%%%%%%%%%%%%%%%%%%%%%%%%%%%%%
\subsection{Stochastic Multicriteria Acceptability Analysis (SMAA)}\label{SMAA_des}%%%
%%%%%%%%%%%%%%%%%%%%%%%%%%%%%%%%%%%%%%%%%%%%%%%%%%%%%%%%%%%%%%%%%%%%%%%%%%%%%%%%%%%%%%
SMAA  ~\cite{Lahdelma,Lahdelma_S2} is a family of MCDA methods which take into account uncertainty or imprecision on the evaluations and preference model parameters. In this section we describe SMAA-2 ~\cite{Lahdelma_S2}, since its underlying preference model is a value function and the Choquet integral preference model belongs to this family. In order to apply SMAA to the level dependent Choquet integral, we shall denote by $\chi$ the evaluation space and by ${\cal M}$ the set of all interval level dependent capacities satisfying the constraints in $E^{DM}$. 

As ROR methods, SMAA methods take into account all the models compatible with the preferences provided by the DM even if in a different way. The indirect preference information is composed of two probability distributions, $f_{\chi}$ and $f_{{\cal M}}$, defined on $\chi$ and ${\cal M}$, respectively. \\
Given $\xi\in\chi$ and $\mu^L\in{\cal M}$, SMAA methods define the rank function 

$$rank(a,\xi,\mu^{L})=1+\sum_{b\in A\setminus\{a\}}\rho\left(Ch^{L}(\xi_b,\mu^L)>Ch^{L}(\xi_a,\mu^{L})\right),$$

\noindent (where $\rho(false)=0$ and $\rho(true)=1$ and $\xi_a$ are the evaluations of $a$ in the matrix $\xi$) that gives the rank position of alternative $a$. On the basis of this rank function, SMAA-2 computes the set of interval level dependent capacities ${\cal M}^{s}(a,\xi)\subseteq{\cal M}$ for which alternative $a$ assumes rank $s=1,2,\ldots,|A|$, as follows:

$${\cal M}^{s}(a,\xi)=\left\{\mu^{L}\in{\cal M}: rank(a,\xi,\mu^L)=s\right\}.$$

\noindent The following indices are therefore computed in SMAA-2:

\begin{itemize}
\item \textit{The rank acceptability index} that measures  the variety of different parameters compatible with the DM's preference information giving to $a$ the rank $s$:

$$
b^{s}(a)=\int_{\xi\in \chi}f_{\chi}(\xi)\int_{\mu^{L}\in{\cal M}^{s}(a,\xi)}f_{{\cal M}}(\mu^{L})\;d\mu^L d\xi;
$$

$b^{s}(a)$ gives the probability that $a$ has rank $s$, and it is within the range $[0,1]$;

\item \textit{The pairwise winning index} \cite{leskinen2006alternatives} that is defined as the frequency with which an alternative $a$ is preferred to an alternative $c$ in the space of the interval level dependent capacities:

$$
p(a,c)=\int_{\mu^L\in {\cal M}} f_{\cal M}(\mu^L) \int_{{\substack{\xi\in \chi: \;Ch^L\left(\xi_{a},\mu^L\right)> Ch^L\left(\xi_{c},\mu^L\right)}}} f_\chi(\xi) \;d\xi d\mu^L.
$$ 

\end{itemize}

%\noindent Let us notice that the previous index $p_{hk}$ is also known as pairwise winning index and it has been introduced in \cite{leskinen2006alternatives,tervonen2009smaa}.

\noindent From a computational point of view, the multidimensional integrals defining the considered indices are estimated by using the Monte Carlo method. In particular, since the preference information provided by the DM and translated into constraints in $E^{DM}$ defines a convex space of parameters (interval level dependent capacities), the Hit-And-Run (HAR) method can be used to sample several of these sets of parameters \cite{smith1984,Tervonen2012EJOR,Tervonen2014}. The application of the level dependent Choquet integral to each of these sets of parameters produces a ranking of the alternatives at hand. Consequently, considering all these rankings simultaneously, the rank acceptability index, as well as the pairwise winning indices described above, can be computed for each alternative and for each pair of alternatives, respectively.

Let us conclude this section by observing that the proposed method could be also used to explain a full ranking provided by the DM and taking into account some preferences on the importance and the interaction between criteria he wishes to express. While, in general, the information provided by the DM regards a set composed of few alternatives and on the basis of this preference information we aim to give recommendations on the whole set of alternatives at hand, in some cases the DM has a clear idea of the ranking of the alternatives and he wishes ``only" to justify them. Technically, his preferences are represented by the importance assigned to the criteria and to the possible interactions between criteria. From a formal point of view, the full ranking of the alternatives given by the DM is translated into inequality constraints between consecutive alternatives in the ranking as described in Section 4.1. For example, if the alternatives at hand are denoted by $a,b,c,d$ and the DM is convinced that they are ranked in this order, then it is necessary to add the following three inequalities: $Ch^{L}(\textbf{a},\mu^{L})\geqslant Ch^{L}(\textbf{b},\mu^{L})+\varepsilon$, $Ch^{L}(\textbf{b},\mu^{L})\geqslant Ch^{L}(\textbf{c},\mu^{L})+\varepsilon$ and $Ch^{L}(\textbf{c},\mu^{L})\geqslant Ch^{L}(\textbf{d},\mu^{L})+\varepsilon$. If there is not any set of parameters compatible with this full preference, then, this means that it is inconsistent. One can therefore check for the cause of the inconsistency as explained in Section 4.1. If, instead, there is at least one set of parameters for which the alternatives' ranking is restored, then there could exist more than one. The application of NAROR and SMAA in this situation is meaningless. Indeed, since all vectors of parameters provide exactly the same ranking (being the one provided by the DM on the whole set of alternatives) the necessary preference relation will give back a total order of the alternatives being the same total order of alternatives the DM already gave. Each alternative $a$ will get $b^{s}(a)=100\%$ in correspondence of the the position $s$ it filled in the ranking given by the DM and, finally, one between $p(a,c)$ and $p(c,a)$ is equal to 100\%, depending on the fact that the DM stated that $a$ is preferred to $c$ or that $c$ is preferred to $a$. Therefore, we could sample a set of vectors of parameters compatible with the preferences provided by the DM and, then, computing its barycenter obtained by averaging, component by component, all sampled vectors. On the basis of the values in the barycenter, one can compute the Shapley indices and the interaction indices giving, therefore, information on the ranking of the criteria with respect to their importance as well as on the possible positive or negative interactions between criteria that could vary from one level of evaluations to the other. Let us observe that the ranking provided by the DM could admit some ex-aequo between the alternatives filling the same position. In this case, the weak preference of $a$ over $c$ is translated into the constraint $Ch^{L}(\textbf{a},\mu^{L})\geqslant Ch^{L}(\textbf{c},\mu^{L})$ (without including $\varepsilon$) and the two pairwise winning indices $p(a,c)$ and $p(c,a)$ would be equal to zero.

%%%%%%%%%%%%%%%%%%%%%%%%%%%%%%%%%%%%%%%%%%%%%%%%%%%%%%%%%%%%%%%%%%%%%%%%%%%%%%%%%%%%%%%%%%%%%%%%%%%%%%%%%%%%%%%%%%
\section{Ordinal regression considering level dependent capacities changing with continuity }\label{Continuity}%%%
%%%%%%%%%%%%%%%%%%%%%%%%%%%%%%%%%%%%%%%%%%%%%%%%%%%%%%%%%%%%%%%%%%%%%%%%%%%%%%%%%%%%%%%%%%%%%%%%%%%%%%%%%%%%%%%%%%

Until now we have considered interval level dependent capacities, that is, level dependent capacities that change with a jump from one interval of evaluations to another but remaining constant within each single interval. Let us observe that this could be an acceptable working hypothesis for real life decision problems in which importance and interaction of criteria change from one interval of evaluations to another. However, it is also reasonable assuming that the importance and the interaction between criteria and, consequently, the level dependent capacities, change with continuity inside each interval of evaluations of considered criteria. Going more in depth in this case, in the following we present a simple and manageable procedure permitting to elicit a level dependent capacity $\mu^L(\cdot,t)$ changing with continuity with respect to $t$. The continuous level dependent capacities that we consider are related again to the partition of the interval $[\alpha,\beta]$ based on the finite sequence $\alpha=a_{0},a_{1},\ldots, a_{p-1},a_{p}=\beta$, that is, the same partition we have considered for the interval level dependent capacities. More precisely, the new continuous level dependent capacity will be defined once the $p+1$ capacities $\mu^L(\cdot,a_0), \mu^L(\cdot,a_1), \ldots,\mu^L(\cdot,a_p)$ are known. These capacities can be induced, for example, from the preference information provided by the DM. For all $E \subseteq G$ and for all $t \in [\alpha,\beta]\setminus\{a_{0},a_{1},\ldots, a_{p-1},a_{p}\}$  we compute $\mu^L(E,t)$ as a linear interpolation between $\mu^L(E,a_{\underline{r}(t)})$ and $\mu^L(E,a_{\underline{r}(t)+1})$, where $\underline{r}(t)=max\{r \in \{0,1,\ldots,p\}: a_r <t\}$,   as follows:   
\begin{equation*} 
\mu^L(E,t)=\mu^L(E,a_{\underline{r}(t)})+\frac{t-a_{\underline{r}(t)}}{a_{\underline{r}(t)+1}-a_{\underline{r}(t)}}\left[\mu^L(E,a_{\underline{r}(t)+1}) - \mu^L(E,a_{\underline{r}(t)})\right].
\end{equation*}
From the equation above we obtain, therefore,
\begin{equation}\label{continuous_capacities}
	\mu^L(E,t)=\frac{a_{\underline{r}(t)+1}-t}{a_{\underline{r}(t)+1}-a_{\underline{r}(t)}}\mu^L(E,a_{\underline{r}(t)})+\frac{t-a_{\underline{r}(t)}}{a_{\underline{r}(t)+1}-a_{\underline{r}(t)}}\mu^L(E,a_{\underline{r}(t)+1}).
\end{equation}
Consequently, the level dependent Choquet integral can be formulated in terms of the $p+1$ capacities $\mu^L(\cdot,a_0), \mu^L(\cdot,a_1), \ldots,\mu^L(\cdot,a_p)$ as follows. Reminding from \cite{greco2011choquet} that
\begin{equation*}
{Ch}^{L}(\mathbf{x},\mu^{L})=\displaystyle \int_\alpha^\beta\mu^L(\{g_i \in G: x_i \geqslant t\},t)dt=\sum_{i=1}^n\int_{x_{(i-1)}}^{x_{(i)}}\mu^L(\{g_i \in G: x_{i} \geqslant t\},t)dt
\end{equation*}
\noindent and observing that, for all $t \in [x_{(i-1)},x_{(i)}[$, 
\begin{equation*}
\{g_i \in G: x_i \geqslant t\}=\{g_i \in G:  x_i \geqslant x_{(i-1)}\}=N_{i}(\mathbf{x})
\end{equation*}
we can write
\begin{equation*}
{Ch}^{L}(\mathbf{x},\mu^{L})=\displaystyle \sum_{i=1}^n\int_{x_{(i-1)}}^{x_{(i)}}\mu^L(N_{i}(\mathbf{x}),t)dt.
\end{equation*}
The values of $\int_{x_{(i-1)}}^{x_{(i)}}\mu^L(N_{i}(\mathbf{x}),t)dt,$ $i=1,\ldots,n$, are computed considering the following two cases:
\begin{itemize}
	\item $\underline{r}(x_{(i-1)})=\underline{r}(x_{(i)})$: in this case, $x_{(i-1)}$ and $x_{(i)}$ belong to the same interval in the partition of $[\alpha,\beta]$. Therefore, $\int_{x_{(i-1)}}^{x_{(i)}}\mu^L(N_{i}(\mathbf{x}),t)dt$ is given by the area of the trapezoid $T$ in Figure \ref{PrimoTrapezio}, that is,
	
\begin{figure}[!h]
\begin{center}
\caption{The value $\int_{x_{(i-1)}}^{x^{(i)}}\mu^{L}(N_{i}(\mathbf{x},t))\;dt$ is given by the area of the trapezoid which vertices are $P_{2}$, $P_{3}$, $P_{5}$ and $P_6$. In this case, the evaluations $x_{(i-1)}$ and $x_{(i)}$ belong to the same interval $\left[a_{\underline{r}\left(x_{(i)}\right)},a_{\underline{r}\left(x_{(i)}\right)+1}\right]$. \label{PrimoTrapezio}}
\includegraphics[scale=0.3]{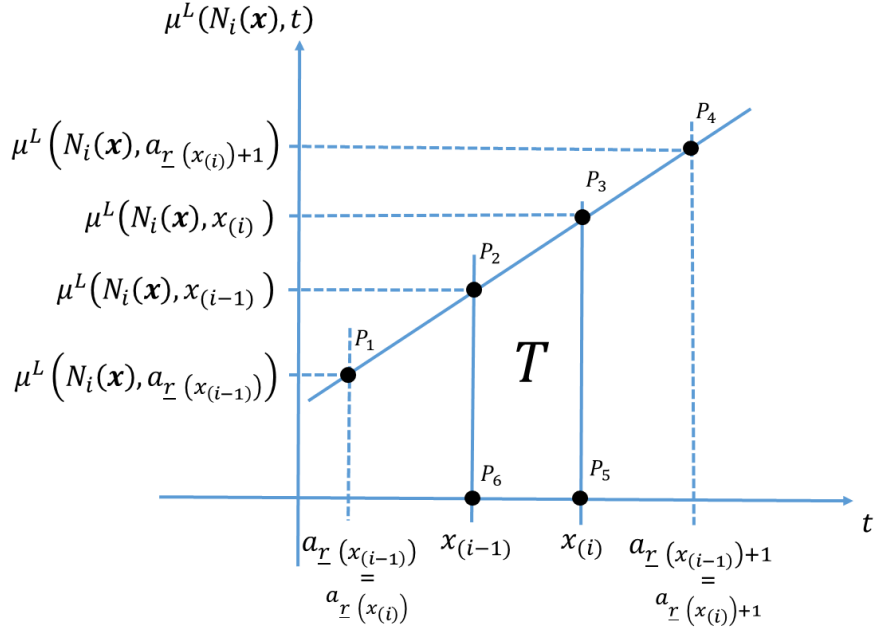}
\end{center}
\end{figure}
	
\begin{equation}\label{ILDChLR}
\int_{x_{(i-1)}}^{x_{(i)}}\mu^L(N_{i}(\mathbf{x}),t)dt=\frac{\left[\mu^L(N_{i}(\mathbf{x}),x_{(i-1)})+\mu^L(N_{i}(\mathbf{x}),x_{(i)})\right]\left(x_{(i)}-x_{(i-1)}\right)}{2}.
\end{equation}
By (\ref{continuous_capacities}) and (\ref{ILDChLR}), $\int_{x_{(i-1)}}^{x_{(i)}}\mu^L(N_{i}(\mathbf{x}),t)dt$  can be expressed in terms of capacities $\mu^L(\cdot,a_0),\ldots,\mu^L(\cdot,a_p)$, as follows:
\begin{eqnarray}\label{FirstTrapezoid}
\int_{x_{(i-1)}}^{x_{(i)}}\mu^L(N_{i}(\mathbf{x}),t)dt&=&\frac{\left[2a_{\underline{r}(x_{(i)})+1}-x_{(i-1)}-x_{(i)}\right]\left(x_{(i)}-x_{(i-1)}\right)}{2\left(a_{\underline{r}(x_{(i)})+1}-a_{\underline{r}(x_{(i)})}\right)}\mu^L\left(N_{i}(\mathbf{x}),a_{\underline{r}(x_{(i)})}\right)+ \nonumber\\
& + & \frac{\left[x_{(i-1)}+x_{(i)}-2a_{\underline{r}(x_{(i)})}\right]\left(x_{(i)}-x_{(i-1)}\right)}{2\left(a_{\underline{r}(x_{(i)})+1}-a_{\underline{r}(x_{(i)})}\right)}\mu^L\left(N_{i}(\mathbf{x}),a_{\underline{r}(x_{(i)})+1}\right);
\end{eqnarray}

\item $\underline{r}(x_{(i-1)})< \underline{r}(x_{(i)})$: in this case, $x_{(i-1)}$ and $x_{(i)}$ do not belong to the same interval in the partition of $\left[\alpha,\beta\right]$. Therefore, the value $\int_{x_{(i-1)}}^{x_{(i)}}\mu^L(N_{i}(\mathbf{x}),t)dt$ is given by the sum of area of trapezoids $T_{\underline{r}(x_{i-1})}, T_{\underline{r}(x_{i-1})+1}, \ldots, T_{\underline{r}(x_{i})}$, in Figure \ref{SecondoTrapezio}:

\begin{footnotesize}
\begin{equation*}
\int_{x_{(i-1)}}^{x_{(i)}}\mu^L(N_{i}(\mathbf{x}),t)dt=\frac{\left[\mu^{L}\left(N_{i}(\mathbf{x}),x_{(i-1)}\right)+\mu^{L}\left(N_{i}(\mathbf{x}),a_{\underline{r}\left(x_{(i-1)}\right)+1}\right)\right]\left(a_{\underline{r}\left(x_{(i-1)}\right)+1}-x_{(i-1)}\right)}{2}+
\end{equation*}
\begin{equation*}
+\sum_{r=\underline{r}(x_{(i-1)})+1}^{\underline{r}(x_{(i)})-1}\frac{\left[\mu^L(N_{i}(\mathbf{x}),a_r)+\mu^L(N_{i}(\mathbf{x}),a_{r+1})\right]\left(a_{r+1}-a_r\right)}{2}+\frac{\left[\mu^L(N_{i}(\mathbf{x}),a_{\underline{r}(x_{(i)})})+\mu^L(N_{i}(\mathbf{x}),x_{(i)})\right]\left(x_{(i)}-a_{\underline{r}(x_{(i)})}\right)}{2}.
\end{equation*}
\end{footnotesize}

\begin{figure}[!h]
\begin{center}
\caption{The value $\int_{x_{(i-1)}}^{x^{(i)}}\mu^{L}(N_{i}(\mathbf{x},t))\;dt$ is given by sum of the area of the trapezoids $T_{\underline{r}(x_{(i-1)})}, T_{\underline{r}(x_{(i-1)})+1}, \ldots, T_{\underline{r}(x_{(i)})}$. In this case, the evaluations $x_{(i-1)}$ and $x_{(i)}$ belong to different intervals in the defined partition of $\left[\alpha,\beta\right]$. \label{SecondoTrapezio}}
\includegraphics[scale=0.3]{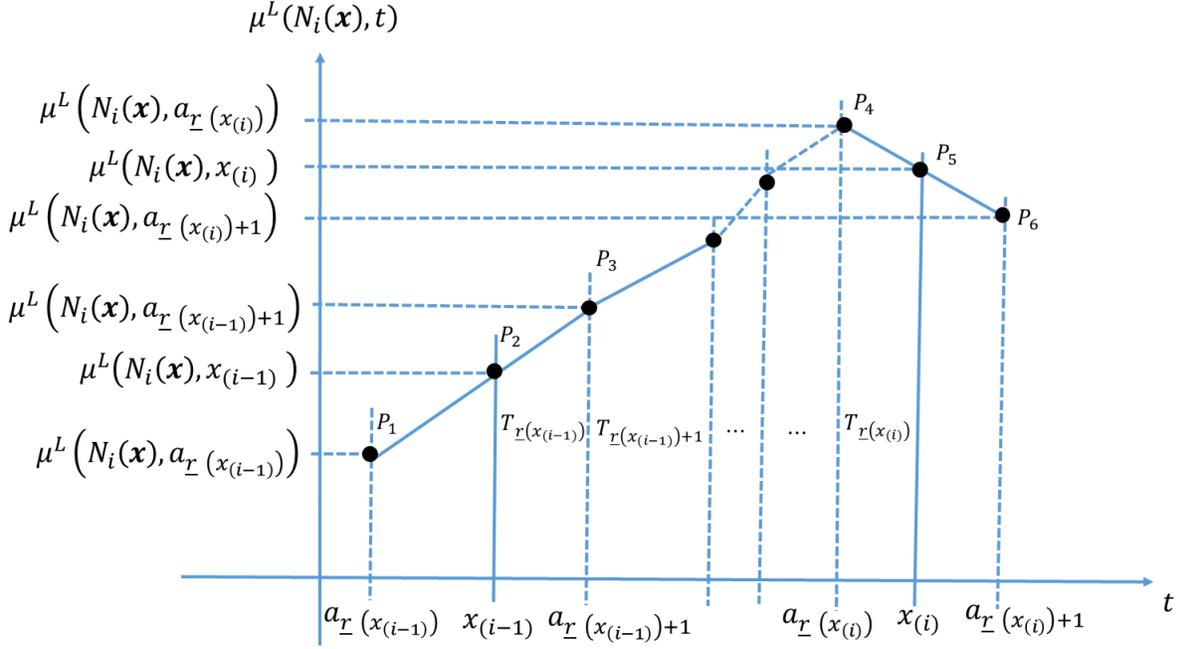}
\end{center}
\end{figure}

In this case, by (\ref{continuous_capacities}), $\int_{x_{(i-1)}}^{x_{(i)}}\mu^L(\{g_i \in G: x_i \geqslant t\},t)dt$  can be expressed in terms of capacities $\mu^L(\cdot,a_0),\ldots,\mu^L(\cdot,a_p)$, as follows: 

\begin{footnotesize}
$$
\int_{x_{(i-1)}}^{x_{(i)}}\mu^L(\{g_i \in G: x_i \geqslant t\},t)dt=\frac{\left(a_{\underline{r}(x_{(i-1)})+1}-x_{(i-1)}\right)^2}{2\left(a_{\underline{r}(x_{(i-1)}+1)}-a_{\underline{r}(x_{(i-1)})}\right)}\mu^L(N_{i}(\mathbf{x}),a_{\underline{r}(x_{(i-1)})})+
$$
$$
+\left[\frac{\left(a_{\underline{r}(x_{(i-1)})+1}+x_{(i-1)}-2a_{\underline{r}(x_{(i-1)})}\right)\left(a_{\underline{r}(x_{(i-1)})+1}-x_{(i-1)}\right)}{2\left(a_{\underline{r}(x_{(i-1)})+1}-a_{\underline{r}(x_{(i-1)})}\right)}+\frac{a_{\underline{r}(x_{(i-1)})+2}-a_{\underline{r}(x_{(i-1)})+1}}{2}\right]\mu^L(N_{i}(\mathbf{x}),a_{\underline{r}(x_{(i-1)})+1})+
$$
$$
+\sum_{r=\underline{r}(x_{(i-1)})+2}^{\underline{r}(x_{(i)})-1}\frac{\left(a_{r+1}-a_{r-1}\right)}{2}\mu^L(N_{i}(\mathbf{x}),a_r)+
$$
$$
+\left[\frac{a_{\underline{r}(x_{(i)})}-a_{\underline{r}(x_{(i)})-1}}{2}+\frac{\left(2a_{\underline{r}(x_{(i)})+1}-a_{\underline{r}(x_{(i)})}-x_{(i)}\right)\left(x_{(i)}-a_{\underline{r}(x_{(i)})}\right)}{2\left(a_{\underline{r}(x_{(i)})+1}-a_{\underline{r}(x_{(i)})}\right)}\right]+\mu^L(N_{i}(\mathbf{x}),a_{\underline{r}(x_{(i)})})+$$
\begin{equation}\label{ILDChLR_neq_}
+\frac{\left(x_{(i)}-a_{\underline{r}(x_{(i)})}\right)^2}{2\left(a_{\underline{r}(x_{(i)})+1}-a_{\underline{r}(x_{(i)})}\right)}\mu^L(N_{i}(\mathbf{x}),a_{\underline{r}(x_{(i)})+1})
\end{equation}	
\end{footnotesize}
\end{itemize}
 
With respect to preference information in terms of importance and interaction between criteria, we can observe that on the basis of (\ref{continuous_capacities}) for all $g_i \in G,$ and for all $t \in [\alpha,\beta]$, the importance index $\phi(i,\mu^L,t)$ and the interaction index $I(E,\mu^L,t)$ can be expressed as linear functions of the same capacities and, consequently, of the related M\"{o}bius transforms. Since, by (\ref{continuous_capacities}), $\mu^L(E,t)$ can be expressed as linear interpolation of $\mu^L(E,a_{\underline{r}(t)})$ and $\mu^L(E,a_{\underline{r}(t)+1})$, $\phi(i,\mu^L,t)$ can be expressed as linear interpolation of $\phi(i,\mu^L,a_{\underline{r}(t)})$ and $\phi(i,\mu^L,a_{\underline{r}(t)+1})$ and $I(E,\mu^L,t)$ can be expressed as linear interpolation of $I(E,\mu^L,a_{\underline{r}(t)})$ and $I(E,\mu^L,a_{\underline{r}(t)+1})$. Indeed,

		$$\phi(i,\mu^L,t)=\phi(i,\mu^L,a_{\underline{r}(t)})+\frac{t-a_{\underline{r}(t)}}{a_{\underline{r}(t)+1}-a_{\underline{r}(t)}}\left[\phi(i,\mu^L,a_{\underline{r}(t)+1}) - \phi(i,\mu^L,a_{\underline{r}(t)})\right]=$$
	\begin{equation*} 
	=\frac{a_{\underline{r}(t)+1}-t}{a_{\underline{r}(t)+1}-a_{\underline{r}(t)}}\phi(i,\mu^L,a_{\underline{r}(t)})+\frac{t-a_{\underline{r}(t)}}{a_{\underline{r}(t)+1}-a_{\underline{r}(t)}}\phi(i,\mu^L,a_{\underline{r}(t)+1})
	\end{equation*}
and 	
	$$I(E,\mu^L,t)=I(E,\mu^L,a_{\underline{r}(t)})+\frac{t-a_{\underline{r}(t)}}{a_{\underline{r}(t)+1}-a_{\underline{r}(t)}}\left[I(E,\mu^L,a_{\underline{r}(t)+1}) - I(E,\mu^L,a_{\underline{r}(t)})\right]=$$
	\begin{equation*} 
	=\frac{a_{\underline{r}(t)+1}-t}{a_{\underline{r}(t)+1}-a_{\underline{r}(t)}}I(E,\mu^L,a_{\underline{r}(t)})+\frac{t-a_{\underline{r}(t)}}{a_{\underline{r}(t)+1}-a_{\underline{r}(t)}}I(E,\mu^L,a_{\underline{r}(t)+1}).
	\end{equation*}
On the basis of this observation, also preference information in terms of importance and interaction of criteria can be represented as linear constraints on the unknowns $m^L(E,a_0),\ldots,m^L(E,a_p)$ for all $E \subseteq G$. For fixing the ideas, let us consider the preference information related to the greater importance of criterion $g_i$ over criterion $g_j$ in the interval of evaluations $[\gamma,\delta]\subseteq [\alpha,\beta]$, denoted by $g_i \succ^{DM}_{[\gamma,\delta]} g_j$. It can be interpreted as  
		\begin{equation}\label{comp_imp} 
	\phi(i,\mu^L,t) > \phi(j,\mu^L,t) \text{ for all } t \in [\gamma,\delta] 
	\end{equation}
that is equivalent to the following constraints
		\begin{equation}\label{comp_imp_LP} 
	\phi(i,\mu^L,t) > \phi(j,\mu^L,t) \text{ for } t=\gamma, a_{\underline{r}(\gamma)+1}, \ldots, a_{\underline{r}(\delta)}, \delta.
	\end{equation}
Indeed, condition (\ref{comp_imp_LP}) is clearly necessary for (\ref{comp_imp}), because it requires that the inequality $\phi(i,\mu^L,t) > \phi(j,\mu^L,t)$ holds for a subset of points in $[\gamma,\delta]$. (\ref{comp_imp_LP}) is also sufficient because independently on the interval among $[\gamma, a_{\underline{r}(\gamma)+1}], [a_{\underline{r}(\gamma)+1},a_{\underline{r}(\gamma)+2}], \ldots, [a_{\underline{r}(\delta)-1},a_{\underline{r}(\delta)}]$ and $[a_{\underline{r}(\delta)}, \delta]$ to which $t$ belongs, $\phi(i,\mu^L,t)$ and $\phi(j,\mu^L,t)$ are convex combinations with the same multipliers of the extremes of the interval. Consequently, if $\phi(i,\mu^L,t)$ is greater than $\phi(j,\mu^L,t)$ in the two extremes of the interval, then $\phi(i,\mu^L,t)$ is greater than $\phi(j,\mu^L,t)$ for all the $t$ in the same interval.

Reasoning in the same way, we can say that the preference information related to the positive (negative) interaction between criteria $g_i$ and $g_j$ in the interval of evaluations $[\gamma,\delta]\subseteq [\alpha,\beta]$ can be interpreted as  
		\begin{equation*}
	I(\mu^L,\{i,j\},t) > [<] 0 \text{ for all } t \in [\gamma,\delta] 
	\end{equation*}
that is equivalent to the  constraint
		\begin{equation}\label{Eq27}
	I(\mu^L,\{i,j\},t) > [<] 0 \text{ for all}\;\; t=\gamma, a_{\underline{r}(\gamma)+1}, \ldots, a_{\underline{r}(\delta)}, \delta.
	\end{equation} 

Let us observe that using (\ref{FirstTrapezoid}) and (\ref{ILDChLR_neq_}), the level dependent Choquet integral  ${Ch}^{L}(\mathbf{x},\mu^{L})$ can be expressed as a linear function of values $\mu^L(E,a_0),\ldots,\mu^L(E,a_p)$ for all $E \subseteq G$ that, in turn, can be expressed in terms of the corresponding M\"{o}bius transforms $m^L(E,a_0),\ldots,m^L(E,a_p)$. Therefore, preference information expressed in terms of pairwise preference comparisons ($a \succ^{DM} b$ or $a \sim^{DM} b$) can be translated into linear constraints in the unknowns $m^L(E,a_0),\ldots,m^L(E,a_p)$ for all $E \subseteq G$. Consequently, results of ROR and SMAA can be obtained by solving LP problems and applying HAR algorithm as explained in Section \ref{NARSMLevDep} for the case of interval level dependent capacities.

%%%%%%%%%%%%%%%%%%%%%%%%%%%%%%%%%%%%%%%%%%%%%%%%%%%%%%%%%%%%%%%%%%%
\subsection{Introductory example: continuous case}\label{inExample}%%%
%%%%%%%%%%%%%%%%%%%%%%%%%%%%%%%%%%%%%%%%%%%%%%%%%%%%%%%%%%%%%%%%%%%
Let us continue the example introduced in Section \ref{ainex} with the aim of illustrating the use of piecewise linear level dependent capacities in the perspective of tempered parsimony for MCDA models we are promoting. For this reason, also in this case, the example is presented quite analytically. The dean maintains the convictions already expressed on the importance and interaction  between Mathematics and Physics. However, he thinks to present the evaluations of students to the school committee using a formal model that avoids jumps of importance and interaction between criteria in the interval of feasible notes. Consequently:  
\begin{itemize}
	\item for notes smaller than 25, Physics is more important than Mathematics and there is a synergy between the two subjects, 
	\item for notes greater than 25, Mathematics is more important than Physics and there is a redundancy between the two subjects.
\end{itemize}
With this aim, the dean considered the continuous level dependent capacity presented in the previous Section \ref{Continuity}. More in detail, the interval $[18,30]$, where the admissible values for the notes lay, was split in the two subintervals $[18,25[$ and $]25,18]$ plus the singleton $\{25\}$, so that the whole level dependent capacity $\mu^L(E,t): 2^{\{Mathematics,Physics\}}\times [18,30]\rightarrow[0,1]$ remains fixed once $\mu^L(E,18), \mu^L(E,25)$ and $\mu^L(E,30)$ are determined, with $E \in 2^{\{Mathematics,Physics\}}$. Let us remember that, in the following equations, we shall write, $M$ and $Ph$ instead of $Mathematics$ and $Physics$. Moreover, to simplify the notation, we shall write $\mu^{L}(M,t)$ and $\mu^{L}(Ph,t)$ instead of $\mu^{L}(\{M\},t)$ and $\mu^{L}(\{Ph\},t)$.\\
Let us observe that:
\begin{itemize}
	\item for notes in the interval $[18,25[$ Mathematics is more important than Physics, that is, $\mu^L(M,t)>\mu^L(Ph,t)$ for all $t \in [18,25[$,
	\item for notes in the interval $]25,30[$ Physics is more important than Mathematics, that is, $\mu^L(Ph,t)>\mu^L(M,t)$ for all $t \in ]25,30]$,
	\item $\mu^L(M,t)$ and $\mu^L(Ph,t)$ have to be continuous on the whole interval $[18,30]$.
\end{itemize}
\noindent Taking into account the previous information, one has to conclude that 
\begin{equation}\label{MEP}
\mu^L(M,25)=\mu^L(Ph,25).
\end{equation} 
Analogously, let us observe that:
\begin{itemize}
	\item for notes in the interval $[18,25[$ there is a redundancy between Mathematics and Physics, that is, 
	$$
	\mu^L\left(\{M,Ph\},t\right)<\mu^L(M,t)+\mu^L(Ph,t), \;\;\mbox{for all}\;t \in [18,25[,
	$$
	\item for notes in the interval $]25,30[$ there is a synergy between notes on Mathematics and Physics, that is, 
	$$
	\mu^L\left(\{M,Ph\},t\right)>\mu^L(M,t)+\mu^L(Ph,t), \;\;\mbox{for all}\;t \in ]25,30],
	$$
	\item $\mu^L\left(\{M,Ph\},t\right)$ has to be continuous on the whole interval $[18,30]$.
\end{itemize}
\noindent On the basis of the previous considerations, one has to conclude that
\begin{equation*}
\mu^L\left(\{M,Ph\},25\right)=\mu^L(M,25)+\mu^L(Ph,25)
\end{equation*}
\noindent so that, observing that $\mu^L(\{M,Ph\},t)=1$ for all $t \in [18,30]$, one has
\begin{equation}\label{MPSum}
\mu^L(M,25)+\mu^L(Ph,25)=1.
\end{equation}
From (\ref{MEP}) and (\ref{MPSum}), we therefore get
\begin{equation}\label{MP25}
\mu^L(M,25)=\mu^L(Ph,25)=0.5.
\end{equation}
Using the continuous level dependent capacity $\mu^L$, according to (\ref{ILDChLR_neq_}), the overall evaluations of students from $A$ to $I$ can be expressed as follows:

\begin{small}
\begin{itemize}
\item ${Ch}^{L}(A,\mu^{L})= 28;$
\item ${Ch}^{L}(B,\mu^{L})= 26 + \mu^{L}(M,25)\frac{(2\times 30-30-26)(30-26)}{2(30-25)} + \mu^{L}(M,30)\frac{(30+26-2\times 25)(30-26)}{2(30-25)} =$
 $$=26 + 1.6\mu^{L}(M,25) + 2.4\mu^{L}(M,30)$$
so that, reminding (\ref{MP25}),
$${Ch}^{L}(B,\mu^{L})=26+1.6 \cdot 0.5 + 2.4 \mu^{L}(M,30)= 26.8 + 2.4 \mu^{L}(M,30);$$ 
\item ${Ch}^{L}(C,\mu^{L})= 26 + \mu^{L}(Ph,25)\frac{(2\times 30-30-26)(30-26)}{2(30-25)} + \mu^{L}(Ph,30)\frac{(30+26-2\times 25)(30-26)}{2(30-25)} =$
$$=26 + 1.6\mu^{L}(Ph,25) + 2.4\mu^{L}(Ph,30)$$
so that, reminding (\ref{MP25}),
$${Ch}^{L}(C,\mu^{L})=26 + 1.6 \cdot 0.5 + 2.4\mu^{L}(Ph,30)=26.8 + 2.4\mu^{L}(Ph,30);$$ 
\item ${Ch}^{L}(D,\mu^{L})= 23;$
\item ${Ch}^{L}(E,\mu^{L})= 21 + \mu^{L}(M,18)\frac{(2\times 25-25-21)(25-21)}{2(25-18)} + \mu^{L}(M,25)\frac{(25 +21 -2\times 18)(25-21)}{2(25-18)} =$
$$ 21 + \frac{8}{7}\mu^{L}(M,18) + \frac{20}{7}\mu^{L}(M,25)$$
so that, reminding (\ref{MP25}),
$${Ch}^{L}(E,\mu^{L})=21+\frac{8}{7}\mu^{L}(M,18)+\frac{20}{7}\cdot 0.5=21+\frac{10}{7}+\frac{8}{7}\mu^{L}(M,18);$$
\item ${Ch}^{L}(F,\mu^{L})= 21 + \mu^{L}(Ph,18)\frac{(2\times 25-25-21)(25-21)}{2(25-18)} + \mu^{L}(Ph,25)\frac{(25 +21 -2\times 18)(25-21)}{2(25-18)} =$
$$ =21 + \frac{8}{7}\mu^{L}(Ph,18) + \frac{20}{7}\mu^{L}(Ph,25)$$
so that, reminding (\ref{MP25}),
$${Ch}^{L}(F,\mu^{L})=21 + \frac{8}{7}\mu^{L}(Ph,18) + \frac{20}{7}\cdot 0.5=21 + \frac{10}{7}+\frac{8}{7}\mu^{L}(Ph,18);$$
\item ${Ch}^{L}(G,\mu^{L})= 26 + \mu^{L}(Ph,25)\frac{(2\times 30-29-26)(29-26)}{2(30-25)} + \mu^{L}(Ph,30)\frac{(29+26- 2\times 25)(29-26)}{2(30-25)} =$
$$=26 + 1.5\mu^{L}(Ph,25) + 1.5\mu^{L}(Ph,30)$$
so that, reminding (\ref{MP25}),
$${Ch}^{L}(G,\mu^{L})=26 + 1.5\cdot 0.5 + 1.5\mu^{L}(Ph,30)= 26,75+ 1.5\mu^{L}(Ph,30);$$
\item ${Ch}^{L}(H,\mu^{L})= 26 + \mu^{L}(M,25)\frac{(2\times 30-29-26)(29-26)}{2(30-25)} + \mu^{L}(M,30)\frac{(29+26-2\times 25)(29-26)}{2(30-25)} =$
$$=26 + 1.5\mu^{L}(M,25) + 1.5\mu^{L}(M,30)$$
so that, reminding (\ref{MP25}),
$${Ch}^{L}(H,\mu^{L})=26 + 1.5\cdot 0.5 + 1.5\mu^{L}(M,30)=26.75 + 1.5\mu^{L}(M,30);$$
\item ${Ch}^{L}(I,\mu^{L})= 27 + \mu^{L}(M,25)\frac{(2\times 30-30-27)(30-27)}{2(30-25)} + \mu^{L}(M,30)\frac{(30+27-2\times 25)(30-27)}{2(30-25)} =$
$$=27 + 0.9\mu^{L}(M,25) + 2.1\mu^{L}(M,30)$$
so that, reminding (\ref{MP25}),
$${Ch}^{L}(I,\mu^{L})=27 + 0.9\cdot 0.5 + 2.1\mu^{L}(M,30)=27.45 + 2.1\mu^{L}(M,30).$$
\end{itemize}
\end{small}
\noindent In consequence, the preferences expressed by the dean are translated into the following constraints in terms of continuous level dependent capacity $\mu^L$:

\begin{equation*}
\left\{
\begin{array}{lllll}
A\succ C & \Rightarrow & {Ch}^{L}(A,\mu^{L}) > {Ch}^{L}(C,\mu^{L}) &\Rightarrow & 28 > 26.8 + 2.4\mu^{L}(Ph,30),\\
C\succ B & \Rightarrow & {Ch}^{L}(C,\mu^{L}) > {Ch}^{L}(B,\mu^{L}) &\Rightarrow & 26.8 + 2.4\mu^{L}(Ph,30) > 26.8 + 2.4\mu^{L}(M,30),\\
E\succ F & \Rightarrow & {Ch}^{L}(E,\mu^{L}) > {Ch}^{L}(F,\mu^{L}) &\Rightarrow & 21+\frac{10}{7}+\frac{8}{7}\mu^{L}(M,18) > 21+\frac{10}{7}+\frac{8}{7}\mu^{L}(Ph,18),\\
F\succ D & \Rightarrow & {Ch}^{L}(F,\mu^{L}) > {Ch}^{L}(D,\mu^{L}) &\Rightarrow & 21+\frac{10}{7}+\frac{8}{7}\mu^{L}(Ph,18) > 23.\\
\end{array}
\right.
\end{equation*}

that is,
\begin{equation}\label{int_good}
A\succ C \Rightarrow 0.5 > \mu^{L}(Ph,30),
\end{equation}
\begin{equation}\label{PhM_good}
C\succ B \Rightarrow \mu^{L}(Ph,30) > \mu^{L}(M,30),
\end{equation}
\begin{equation}\label{PhM_bad}
E\succ F \Rightarrow \mu^{L}(M,18)  > \mu^{L}(Ph,18),
\end{equation}
\begin{equation}\label{int_bad}
F\succ D \Rightarrow \mu^{L}(Ph,18) > 0.5.
\end{equation}

Observe that these inequalities obtained from the preferences expressed by the dean are in agreement with his considerations about the importance and the interaction of criteria. Indeed, from (\ref{int_good})-(\ref{int_bad}) we get what follows:
\begin{itemize}
	\item by (\ref{continuous_capacities}) and reminding (\ref{MP25}), for all $t \in ]25,30]$ we obtain
	\begin{equation}\label{Ph_good}
	\mu^L(Ph,t)=\frac{30-t}{30-25}\mu^L(Ph,25)+\frac{t-25}{30-25}\mu^L(Ph,30)=\frac{30-t}{5} \cdot 0.5+\frac{t-25}{5}\mu^L(Ph,30)
\end{equation}
and 
\begin{equation}\label{M_good}
	\mu^L(M,t)=\frac{30-t}{30-25}\mu^L(M,25)+\frac{t-25}{30-25}\mu^L(M,30)=\frac{30-t}{5} \cdot 0.5+\frac{t-25}{5}\mu^L(M,30)
\end{equation}
so that, by (\ref{PhM_good}), we get
\begin{equation}\label{Ph_good_}
	\mu^L(Ph,t)>\mu^L(M,t)\;\; \mbox{for all}\;\; t \in ]25,30],
\end{equation}
confirming the greater importance of Physics over Mathematics for notes greater than 25; 
\item putting together (\ref{int_good}) and (\ref{Ph_good}), we obtain
\begin{equation*}
0.5 > \mu^{L}(Ph,t) \mbox{ for all } t \in ]25,30],
\end{equation*}
so that, considering also (\ref{Ph_good_}) and reminding that $\mu^L(\{M,Ph\},t)=1$ for all $t$, we get 
\begin{equation*}
\mu^{L}(Ph,t) + \mu^{L}(M,t) <\mu^L(\{M,Ph\},t) \mbox{ for all } t \in ]25,30],
\end{equation*}
confirming the synergy between Mathematics and Physics for notes greater than 25; 
\item by (\ref{continuous_capacities}) and (\ref{MP25}), for all $t \in [18,25[$ we obtain
	\begin{equation}\label{Ph_bad__}
	\mu^L(Ph,t)=\frac{25-t}{25-18}\mu^L(Ph,18)+\frac{t-18}{25-18}\mu^L(Ph,25)=\frac{25-t}{7}\mu^L(Ph,18) +\frac{t-18}{7} \cdot 0.5
\end{equation}
and 
\begin{equation*}
	\mu^L(M,t)=\frac{25-t}{25-18}\mu^L(M,18)+\frac{t-18}{25-18}\mu^L(M,25)=\frac{25-t}{7}\mu^L(M,18) +\frac{t-18}{7}\cdot 0.5
\end{equation*}
so that, by (\ref{PhM_bad}), we get
\begin{equation}\label{Ph_bad_}
	\mu^L(M,t)>\mu^L(Ph,t) \mbox{ for all } t \in [18,25[,
\end{equation}
confirming the greater importance of Mathematics  over Physics for notes smaller than 25; 
\item putting together (\ref{int_bad}) and (\ref{Ph_bad__}), we obtain
\begin{equation*}
 \mu^{L}(Ph,t) > 0.5 \mbox{ for all } t \in ]25,30],
\end{equation*}
so that, considering also (\ref{Ph_bad_}) and reminding that $\mu^L(\{M,Ph\},t)=1$ for all admissible $t$, we get 
\begin{equation*}
\mu^{L}(Ph,t) + \mu^{L}(M,t) >\mu^L(\{M,Ph\},t) \mbox{ for all } t \in [18,25[,
\end{equation*}
confirming the redundancy between Mathematics and Physics for notes smaller than 25. 
\end{itemize}
Since the preferences expressed by the dean on the students are in agreement with the importance and interaction between criteria we wanted to take into consideration, the compatible level dependent capacities $\mu^L$ are those satisfying conditions (\ref{MP25}) and (\ref{int_good})-(\ref{int_bad}).
    
Now let us compare between them students $G, H$ and $I$ through the  Choquet integral based on the piecewise linear compatible capacities $\mu^L$ compatible with the preferences expressed by the dean. Let us observe that the three students have notes greater than 25 both on Mathematics and Physics, so that the different evaluations provided by the Choquet integral depend only on the values $\mu^L(E,t), E\subseteq \{M,Ph\}$ and $t\in ]25,30]$. Regarding the piecewise linear compatible capacities, they are determined only by the values of $\mu^L(M,30)$ and $\mu^L(Ph,30)$. Indeed, $\mu^L(\emptyset,t)=0$ and $\mu^L(\{M,Ph\},t)=1$ for all admissible value of $t$, while according to (\ref{Ph_good}) and (\ref{M_good}), for all $t \in ]25,30]$, the values of $\mu^L(M,t)$  and $\mu^L(Ph,t)$ depend only on $\mu^L(M,30)$ and $\mu^L(Ph,30)$. Therefore, with respect to students $G, H$ and $I$, taking into account preferences expressed by the dean and translated into the constraints (\ref{PhM_good}) and (\ref{M_good}), the set of all compatible piecewise linear capacities are represented by the points inside the triangle which vertices are $O=(0,0), P_1=(0,0.5)$ and $P_3=(0.5,0.5)$ in the $\mu^L(M,t) - \mu^L(Ph,t)$ plane in Figure \ref{ProbabilityGI}. 
%on the values of on the basis of \ref{continuous_capacities} and \ref{MP25}
%\begin{equation}\label{continuous_capacities_M}
	%\mu^L(M,t)=\frac{30-t}{30-25}\mu^L(M,25)+\frac{t-25}{30-25}\mu^L(M,30)=\frac{30-t}{30-25}\cdot 0.5 +\frac{t-25}{30-25}\mu^L(M,30)
%\end{equation}
%and 
%\begin{equation}\label{continuous_capacities_Ph}
	%\mu^L(Ph,t)=\frac{30-t}{30-25}\mu^L(Ph,25)+\frac{t-25}{30-25}\mu^L(Ph,30)=\frac{30-t}{30-25}\cdot 0.5 +\frac{t-25}{30-25}\mu^L(Ph,30)
%\end{equation}
%while, $\mu(\emptyset,t)=0$ and $\mu(\{M,Ph\},t)=1$ for all admissible value of $t$.

On the basis of the above remarks, comparing students $G, H$ and $I$, we have that:
\begin{itemize}
	\item as in the case of interval of level dependent capacity (and, in fact, as in the case of any multiple criteria aggregation procedure) the dominance of $I$ over $H$ ensures that $I$ is preferred to $H$ for any compatible piecewise linear dependent capacity, that is, there is a necessary preference of $I$ over $H$;
	\item since $G \succ H$ is equivalent to 
$${Ch}^{L}(G,\mu^{L})>{Ch}^{L}(H,\mu^{L}),$$
that is,
$$26.75 + 1.5\mu^{L}(Ph,30)>26.75 + 1.5\mu^{L}(M,30),$$
by (\ref{PhM_good}), we can conclude that $G$ is preferred to $H$ for all compatible piecewise linear dependent capacities, so that there is a necessary preference of $G$ over $H$; 
\item since $G \succ I$ and $I \succ G$ are equivalent to 
$${Ch}^{L}(G,\mu^{L})>{Ch}^{L}(I,\mu^{L}) \mbox{ and } {Ch}^{L}(I,\mu^{L})>{Ch}^{L}(G,\mu^{L}),$$
that is,
$$26.75 + 1.5\mu^{L}(Ph,30)>27.45 + 2.1\mu^{L}(M,30) \mbox{ and }  27.45 + 2.1\mu^{L}(M,30)>26.75 + 1.5\mu^{L}(Ph,30),$$
there are compatible piecewise linear level dependent capacities $\mu^L$ for which $G$ is preferred to $I$ (for example, considering $\mu^L$ such that $\mu^L(M,30)=\frac{1}{126}$ and $\mu^L(Ph,30)=\frac{22}{45}$ representing point $B_{G \succsim I}$ in Figure \ref{ProbabilityGI} for which ${Ch}^{L}(G,\mu^{L})=27.4833$ and ${Ch}^{L}(I,\mu^{L})=27.4667$) and compatible piecewise linear level dependent capacities $\mu^L$ for which $I$ is preferred to $G$ (for example, considering $\mu^L$ such that $\mu^L(M,30)=\frac{11}{84}$ and $\mu^L(Ph,30)=\frac{11}{30}$ representing point $B_{I \succsim G}$ in Figure \ref{ProbabilityGI} for which ${Ch}^{L}(G,\mu^{L})=26.9464$ and ${Ch}^{L}(I,\mu^{L})=28.22$). Therefore, there is a a possible preference of $G$ over $I$ as well as a possible preference of $I$ over $G$ and there is no necessary preference between these two students. 		
\end{itemize}
More precisely, with respect to the comparison of students $G$ and $I$, we have that: 
\begin{itemize}
	\item $G$ is preferred over $I$ for the piecewise linear level dependent capacities $\mu^L$ represented by the points in the triangle of vertices $P_1\equiv\left(0,\frac{1}{2}\right), P_2\equiv(\frac{1}{42},\frac{1}{2})$ and $P_5\equiv(0,\frac{7}{15})$ (the above mentioned point $B_{G \succsim I}$ is the barycentre of this triangle),
	\item $I$ is preferred over $G$ for the piecewise linear level dependent capacities $\mu^L$ represented by the points in the polygon of vertices $O\equiv(0,0), P_3\equiv\left(\frac{1}{2},\frac{1}{2}\right), P_2\equiv\left(\frac{1}{42},\frac{1}{2}\right)$ and $P_5\equiv\left(0,\frac{7}{15}\right)$ (the above mentioned point $B_{I \succsim G}$ is the barycentre of this polygon).
\end{itemize}
Consequently the probability that taking randomly a compatible piecewise linear level dependent capacity $\mu^L$, $G$ is preferred over $I$ is 0.3175\%, corresponding to the ratio between the area of the triangle of vertices $P_1, P_2$ and $P_5$, and the triangle of vertices $O, P_1, P_3$, while the probability that $I$ is preferred over $G$ is 99.6825\%, corresponding to the the ratio between the area of the polygon of vertices $O, P_3, P_2$ and $P_5$, and the triangle of vertices $O$, $P_1$ and $P_3$.  Considering the probability of being the first, the second or the third taking randomly a compatible piecewise linear level dependent capacity $\mu^L$, we have that: 
\begin{itemize}
	\item as $G$ and $I$ are preferred to $H$ for all  compatible $\mu^L$, there is a probability of 100\% that $H$ is the third,
	\item as $I$ is the first when he is preferred to $G$ and the second otherwise, there is a probability of 0.3175\% that $H$ is the first and a probability of 99.6825\% that he is the second,
	\item finally, $G$ is the first with a probability of 99.6825\% and the second with a probability of 0.3175\%.
\end{itemize}
%$G$ is preferred over $I$ for the level dependent capacities $\mu^L$ represented by the triangle of vertices $P_1=(0,0.5), P_2=(\frac{1}{42},0.5)$ and $P_3=(0,\frac{7}{15})$. one can get a more precise idea of their preferability by computing the probability that taking randomly a compatible piecewise linear level dependent capacities $\mu^L$, $I$ is preferred to $G$ or $G$ is preferred to $I$. Since the     
%Now, with respect to the preferences provided for the importance an interaction between criteria, we have the following further constraints:\\
%
%\begin{description}
%\item\begin{equation}\label{ContCrit2}
%\left\{
%\begin{array}{l}
%\mu^{L}(Ph,30)\geq \mu^{L}(M,30),\\
%\mu^{L}(Ph,25)\geq \mu^{L}(M,25),\\
%\mu^{L}(Ph,18)\leq \mu^{L}(M,18),\\
%\mu^{L}(Ph,25)\leq \mu^{L}(M,25),\\
%\end{array}
%\right.
%\end{equation}
%
%\item\begin{equation}\label{ContCrit3}
%\left\{
%\begin{array}{l}
%1-\mu^{L}(Ph,30)-\mu^{L}(M,30)\geq 0,\\
%1-\mu^{L}(Ph,25)-\mu^{L}(M,25)\geq 0.\\
%1-\mu^{L}(Ph,18)-\mu^{L}(M,18)\geq 0,\\
%1-\mu^{L}(Ph,25)-\mu^{L}(M,25)\geq 0.\\
%\end{array}
%\right.
%\end{equation}
%\end{description}
%
%Note that, to be eq. \ref{ContCrit2} and \ref{ContCrit3} feasible, we must have that $\mu^{L}(Ph,25)=\mu^{L}(M,25)=0,5$. 
%
%Note that, since from the dean preference information about the importance of criterion \\
%$$\mu^{L}(Ph,[26,30]) > \mu^{L}(M,[26,30])\;\;\mbox{and}\;\;\mu^{L}(M,[18,25]) > \mu^{L}(Ph,[18,25])$$ \\
%\noindent the continuous function $\mu^{L}(Ph,t) - \mu^{L}(M,t)$ must have at least a root for $t\in ]25,26[$ 

\begin{figure}[!h]
\begin{center}
\caption{The probability that $G$ is preferred to $I$ is equal to the ratio between the area of the triangle $P_{1}P_{2}P_5$ and the area of the triangle $OP_1P_3$. Analogously, the probability that $I$ is preferred to $G$ is equal to the ratio between the area of the polygon $OP_3P_2P_5$ and the are of the triangle $OP_1P_3$. \label{ProbabilityGI}}
\includegraphics[scale=0.3]{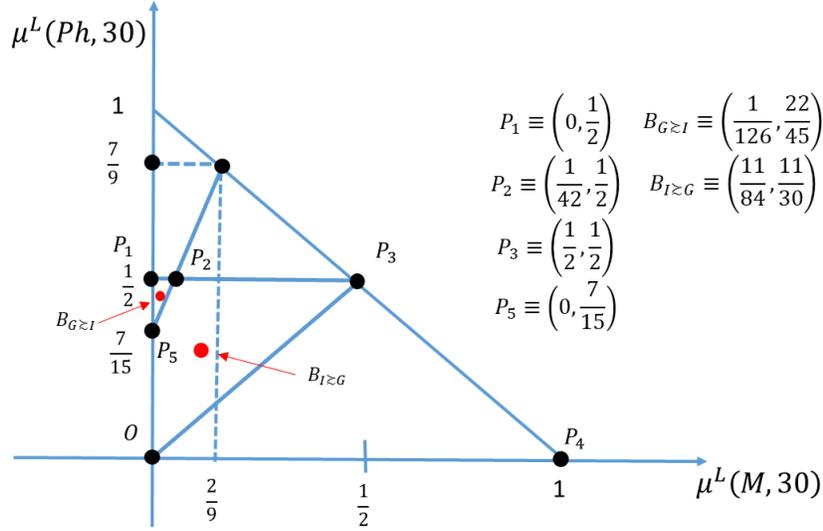}
\end{center}
\end{figure}

Let us observe that the considerations on the participation of the DM to the fixing of the reference points that we did at the end of Section \ref{ainex} for the interval level dependent capacities hold also for the piecewise linear level dependent capacities. Let us also point out that, in agreement with what observed at the end of the previous subsection, we have handled the application of the Choquet integral with respect to piecewise linear level dependent capacity using linear constraints only. This confirms that a systematic application of ROR and SMAA to this model can be dealt by solving LP problems and using the HAR algorithm.

%%%%%%%%%%%%%%%%%%%%%%%%%%%%%%%%%%%%%%%%%%%%%%%%
\section{Illustrative example}\label{Example}%%%
%%%%%%%%%%%%%%%%%%%%%%%%%%%%%%%%%%%%%%%%%%%%%%%%
In this section, we shall apply our methodology to a real world problem attracting the interest of experts and public opinion in the last years, that is the university rankings. As acknowledged in \cite{CorrenteGrecoSlowinski2019}, nowadays there are plenty of university rankings, starting from the most known one, that is the Shanghai ranking (ARWU) \cite{ShanghaiData}, to the ranking provided by The Times \cite{THE}. All of them are based on different aggregation methods and different importance assigned to the criteria taken into account in the considered analysis. \\
In this case, we will concentrate our attention on the comparison of economics universities all over the world evaluated with respect to the teaching and learning (TL) aspect of their bachelor and master courses. The evaluations of the universities are measured by means of 4 subcriteria being Bachelor Graduation Rate (BGR), Masters Graduation Rate (MGR), Bachelors Graduation on Time (BGT) and Masters Graduation on Time (MGT). The description of the four criteria is given in Table \ref{CriteriaTable}. 

\begin{table}[!h]
  \centering
  \caption{Description of the criteria\label{CriteriaTable}}
	\resizebox{1\textwidth}{!}{
    \begin{tabular}{ll}
		\textbf{Criterion} & \textbf{Description} \\
		\hline 
		\hline
     Bachelor Graduation Rate (BGR) & The percentage of new entrants that successfully completed their bachelor programme. \\
		 Masters Graduation Rate (MGR) &  The percentage of new entrants that successfully completed their master programme. \\
		 Bachelors Graduation on Time (BGT) & The percentage of graduates that graduated within the time expected (normative time) \\
		& for their bachelor programme. \\
		 Masters Graduation on Time (MGT) & The percentage of graduates that graduated within the time expected (normative time) \\
		& for their masters programme.\\
		\hline
		\hline
	    \end{tabular}%
			}
\end{table}%

The database used in this study is U-Multirank \cite{multirank} where the evaluations can vary between 1 (Weak) to 5 (Very Good). For the sake of simplicity we are supposing that the ordinal evaluations from 1 to 5 are accepted by the DM as the corresponding utility value (for two procedures to assess utility values see \cite{angilella2015,BotteroEtAl2018}). At the beginning, 1609 universities are considered. Removing those having unavailable data for at least one of the four criteria and those for which the considered aspect was not applicable, 501 universities were left. Finally, we considered only universities that are not presenting the same profile. In this way, the considered universities are 124. The methodology has been applied to the whole set of universities and the final results can be consulted clicking on the following link: \href{http://www.antoniocorrente.it/wwwsn/images/allegati_articoli/ChLevelDependentNARORSMAA.xlsx}{supplementary results}. Anyway, for space reasons we show here the results of the application performed on a subset of universities that have been selected as explained in the following. Since the most interesting comparison regards the one performed between non-dominated universities, following the same procedure used in the celebrated evolutionary algorithm NSGA-II \cite{debieee}, we ordered the universities in non-dominated fronts. This means that we put in the first front the non-dominated universities. After removing them, we put in the second front the non-dominated universities among the remaining ones, and so on. Finally, 13 different fronts have been built. To show the application of the method, we have chosen the biggest one, that is the 7th, composed of 19 different universities which performances on the four criteria are shown in Table \ref{Sample}. Let us observe that another front or other universities (for example those belonging to the same Country) could be chosen. Close to these performances, we report also the performances of few universities the DM knows well and on which he provided some preference information. For brevity, the universities have been labeled by $U_{h}$. Anyway, in the file reporting the results of the application of the proposed method to the whole set of universities, we denoted the universities with their name. 

\begin{table}[!h]
\begin{center}
\caption{Performances of universities on the Teaching and Learning aspect}\label{SMAAResults}
\subtable[Sample of the universities\label{Sample}]{%
\resizebox{0.3\textwidth}{!}{
 \begin{tabular}{l|cccc}    
 & $BGR$ & $MGR$ & $BGT$ & $MGT$ \\
\hline
\hline
$U_{1}$  &  3     & 2     & 5     & 4 \\        
$U_{2}$  &  3     & 3     & 4     & 4 \\
$U_{3}$  &  2     & 4     & 4     & 4 \\
$U_{4}$  &  4     & 4     & 2     & 4 \\
$U_{5}$  &  4     & 4     & 4     & 2 \\    
$U_{6}$  &  4     & 4     & 3     & 3 \\    
$U_{7}$  &  5     & 2     & 4     & 3 \\    
$U_{8}$  &  3     & 3     & 5     & 3 \\    
$U_{9}$  &  3     & 4     & 4     & 3 \\    
$U_{10}$ &  5     & 3     & 3     & 2 \\    
$U_{11}$ &  2     & 2     & 5     & 5 \\    
$U_{12}$ &  2     & 3     & 5     & 4 \\    
$U_{13}$ &  2     & 3     & 4     & 5 \\    
$U_{14}$ &  4     & 2     & 4     & 4 \\    
$U_{15}$ &  4     & 3     & 4     & 3 \\    
$U_{16}$ &  3     & 2     & 4     & 5 \\    
$U_{17}$ &  3     & 3     & 2     & 5 \\    
$U_{18}$ &  3     & 4     & 3     & 4 \\    
$U_{19}$ &  5     & 4     & 2     & 3 \\
    \end{tabular}%
}
}
\subtable[Universities on which the DM provided his preferences\label{PREFUniversities}]{%
\resizebox{0.3\textwidth}{!}{
\begin{tabular}{l|cccc}    
 & $BGR$ & $MGR$ & $BGT$ & $MGT$ \\
\hline
\hline
$U_{20}$  &  5     & 5     & 5     & 4 \\
$U_{21}$  &  3     & 2     & 2     & 3 \\
$U_{22}$  &  5     & 4     & 4     & 5 \\
$U_{23}$  &  3     & 3     & 3     & 2 \\    
$U_{24}$  &  5     & 4     & 5     & 4 \\    
$U_{25}$  &  4     & 5     & 5     & 4 \\    
$U_{26}$  &  2     & 3     & 3     & 3 \\    
$U_{27}$  &  2     & 3     & 3     & 2 \\    
$U_{28}$  &  3     & 2     & 3     & 2 \\    
$U_{29}$  &  4     & 5     & 5     & 5 \\
\end{tabular}%
}
}
\end{center}
\end{table}

At first, the DM provided some preferences between criteria as well as some possible positive and negative interactions between pairs of criteria. In particular, he considered as benchmark ``not so good universities", being those with performances varying between 1 and 3 and ``good universities" being those with performances between 4 and 5. In this way, the interval scale $[1,5]$ is split in two subintervals being $[1,3]$ and $]3,5]$. The DM expressed the willingness to focus the attention on the bachelor students in case of not so good evaluations on considered criteria, while he wants to give more importance to master students in evaluating universities presenting good performances. This perspective is represented in formal terms as follows:  

\begin{itemize}
\item in evaluating universities with not so good evaluations on considered criteria, $BGR$ is more important than $MGR$, while the vice versa is true in evaluating universities presenting good performances. This preference information is translated into the following constraints:
\begin{itemize}
\item $\phi\left(\mu^{L},\{BGR\},[1,3]\right)>\phi\left(\mu^{L},\{MGR\},[1,3]\right)$,
\item $\phi\left(\mu^{L},\{BGR\},]3,5]\right)<\phi\left(\mu^{L},\{MGR\},]3,5]\right)$;
\end{itemize}
\item in evaluating universities with not so good evaluations, $BGT$ is more important than $MGT$, while the viceversa is true in evaluating universities presenting good performances. This preference information is translated into the following constraints:
\begin{itemize}
\item $\phi\left(\mu^{L},\{BGT\},[1,3]\right)>\phi\left(\mu^{L},\{MGT\},[1,3]\right)$,
\item $\phi\left(\mu^{L},\{BGT\},]3,5]\right)<\phi\left(\mu^{L},\{MGT\},]3,5]\right)$;
\end{itemize}
\item $BGR$ and $BGT$ are positively interacting as well as $MGR$ and $MGT$ in evaluating universities with not so good performances on considered criteria. The same pairs of criteria are negatively interacting in evaluating universities with good performances. On one hand, the DM would like to give a bonus to not so good universities in which students (both bachelor and master) are able not only to complete their programmes but also to finish their studies in the expected time. On the other hand, since good universities are, in general, attended by good students, it is not so surprising that students would complete their programmes and would be also able to finish their studies in the expected time. For this reason, in this case, the considered pairs of criteria are negatively interacting. This piece of preference information is translated into the following constraints: 
\begin{itemize}
\item $I\left(\mu^{L},\{BGR,BGT\},[1,3]\right)>0$, 
\item $I\left(\mu^{L},\{BGR,BGT\},]3,5]\right)<0$,
\item $I\left(\mu^{L},\{MGR,MGT\},[1,3]\right)>0$,
\item $I\left(\mu^{L},\{MGR,MGT\},]3,5]\right)<0$.  
\end{itemize}
\end{itemize}

\noindent The DM provided also the following preference information on the universities which performances are shown in Table \ref{PREFUniversities}:

$$
\begin{array}{llll}
P1) & U_{28}\succ U_{21} & \mbox{and} & U_{22}\succ U_{24}, \\
P2) & U_{28}\succ U_{27} & \mbox{and} & U_{25}\succ U_{24}, \\
P3) & U_{23}\succ U_{26} & \mbox{and} & U_{29}\succ U_{20}. \\
\end{array}
$$

\noindent To represent this preference information and, in particular, the fact that the importance of criteria as well as the interactions between criteria are dependent on the considered performances we have to use the level dependent Choquet integral. As already explained above, we consider the partition $a_0=1,$ $a_2=3,$ $a_2=5$ of the interval $[1,5]$. This means that we are taking into account two capacities, one on the interval $[1,3]$ and one on the interval $]3,5]$. Considering this partition, we solve the LP problem (\ref{LPExistence}) presented in Section \ref{NAROR_des}, obtaining $\varepsilon^{*}=0.25>0$.  This means that the level dependent Choquet integral, with the considered partition, is able to represent the preferences provided by the DM. 

Since the universities at hand belong to the same non-dominated front, the dominance relation is empty and, consequently, looking only at their performances nothing can be said. Consequently, as following step, we decided to apply the NAROR to check if the necessary and possible preference relations can give a more clear view of how the universities can be compared each other. Unlikely, the necessary preference relation is empty (apart from the trivial relations $U_{h}\succsim U_{h}$ for all $U_{h}$). This means that there is not a certain recommendation on the problem at hand since there are models compatible with the DM's preferences for which each university can be preferred to any other and vice versa. For such a reason, as described in Section \ref{SMAA_des}, we apply the SMAA methodology obtaining probabilistic preferences. At first, we computed the rank acceptability indices of each university.

\begin{table}[!h]
\begin{center}
\caption{Rank Acceptability indices for the considered universities}\label{RAIResults}
\resizebox{1\textwidth}{!}{
 \begin{tabular}{l|ccccccccccccccccccc}
 & $b^{1}(\cdot)$  & $b^{2}(\cdot)$ & $b^{3}(\cdot)$ & $b^{4}(\cdot)$ & $b^{5}(\cdot)$ & $b^{6}(\cdot)$ & $b^{7}(\cdot)$ & $b^{8}(\cdot)$ & $b^{9}(\cdot)$ & $b^{10}(\cdot)$ & $b^{11}(\cdot)$ & $b^{12}(\cdot)$ & $b^{13}(\cdot)$ & $b^{14}(\cdot)$ & $b^{15}(\cdot)$ & $b^{16}(\cdot)$ & $b^{17}(\cdot)$ & $b^{18}(\cdot)$ & $b^{19}(\cdot)$ \\
\hline
\hline
 $U_{1}$ &   0.261 & 5.511 & 10.258 & 11.567 & 10.474 & 7.405 & 6.247 & 5.365 & 6.219 & 5.29  & 5.227 & 5.156 & 4.837 & 4.231 & 4.486 & 2.852 & 3.194 & 1.157 & 0.263 \\
 $U_{2}$ &   0     & 0.221 & 1.448 & 5.285 & 9.186 & 12.281 & 13.971 & 13.33 & 12.842 & 9.735 & 9.824 & 6.901 & 3.388 & 0.987 & 0.561 & 0.04  & 0     & 0     & 0 \\
 $U_{3}$ &   0.001 & 0.013 & 0.098 & 0.228 & 0.555 & 0.823 & 0.795 & 1.151 & 1.884 & 2.156 & 2.545 & 3.939 & 5.492 & 7.25  & 11.736 & 16.019 & 16.734 & 15.598 & 12.983 \\
 $U_{4}$ &   0.45  & 1.852 & 2.335 & 1.683 & 2.719 & 2.583 & 3.106 & 3.748 & 3.968 & 4.863 & 4.541 & 5.329 & 6.314 & 6.468 & 8.35  & 9.766 & 12.042 & 11.055 & 8.828 \\
 $U_{5}$ &   11.409 & 8.189 & 8.489 & 8.105 & 8.125 & 6.906 & 6.405 & 6.381 & 5.205 & 4.287 & 4.4   & 4.198 & 4.872 & 5.222 & 3.994 & 2.44  & 0.964 & 0.409 & 0 \\
 $U_{6}$ &   0.156 & 0.971 & 3.005 & 5.686 & 5.5   & 6.283 & 6.032 & 5.909 & 5.958 & 7.523 & 7.421 & 7.51  & 8.734 & 9.421 & 7.345 & 6.571 & 5.547 & 0.397 & 0.031 \\
 $U_{7}$ &   4.925 & 3.732 & 5.354 & 6.485 & 7.326 & 6.617 & 5.931 & 5.843 & 4.752 & 4.337 & 4.504 & 5.032 & 5.056 & 6.198 & 7.919 & 7.65  & 4.699 & 2.852 & 0.788 \\
 $U_{8}$ &   2.165 & 3.094 & 4.395 & 5.777 & 6.566 & 6.134 & 5.965 & 5.845 & 5.678 & 5.37  & 5.83  & 5.823 & 6.939 & 7.416 & 6.837 & 6.163 & 4.64  & 3.421 & 1.942 \\
 $U_{9}$ &   1.125 & 3.865 & 4.489 & 5.659 & 5.754 & 6.034 & 6.983 & 7.679 & 8.133 & 9.077 & 10.572 & 11.091 & 9.294 & 6.573 & 2.707 & 0.859 & 0.106 & 0     & 0 \\
 $U_{10}$ &   0.002 & 0.172 & 0.216 & 0.175 & 0.366 & 0.273 & 0.399 & 0.709 & 1.29  & 1.914 & 2.046 & 2.16  & 2.707 & 4.187 & 6.327 & 10.301 & 14.982 & 17.199 & 34.575 \\
 $U_{11}$ &   10.098 & 15.65 & 15.175 & 8.586 & 5.391 & 3.628 & 3.652 & 2.8   & 3.178 & 3.021 & 3.333 & 3.424 & 3     & 3.662 & 3.735 & 2.622 & 1.932 & 3.936 & 3.177 \\
 $U_{12}$ &   0     & 0     & 0.06  & 0.385 & 0.746 & 2.378 & 2.871 & 3.163 & 3.294 & 3.67  & 3.97  & 5.435 & 6.117 & 6.853 & 7.994 & 10.135 & 12.956 & 14.737 & 15.236 \\
 $U_{13}$ &   2.881 & 14.07 & 13.639 & 9.138 & 6.819 & 5.829 & 4.784 & 4.147 & 3.979 & 4.195 & 3.783 & 3.955 & 4.597 & 4.541 & 3.964 & 4.899 & 3.477 & 1.075 & 0.228 \\
 $U_{14}$ &   1.478 & 14.957 & 10.974 & 9.822 & 8.362 & 7.491 & 6.417 & 5.906 & 4.77  & 5.028 & 4.337 & 4.8   & 4.625 & 4.565 & 3.436 & 1.962 & 0.543 & 0.225 & 0.302 \\
 $U_{15}$ &   0.006 & 0.177 & 1.07  & 3.279 & 6.106 & 8.651 & 10.31 & 10.36 & 12.741 & 13.109 & 10.313 & 9.528 & 6.599 & 4.574 & 2.677 & 0.451 & 0.049 & 0     & 0 \\
 $U_{16}$ &   60.842 & 17.604 & 7.103 & 3.182 & 2.226 & 1.239 & 0.964 & 1.042 & 1.095 & 1.104 & 1.031 & 0.881 & 0.984 & 0.578 & 0.125 & 0     & 0     & 0     & 0 \\
 $U_{17}$ &   1.109 & 4.534 & 4.778 & 7.42  & 4.529 & 3.817 & 3.886 & 3.516 & 3.06  & 3.545 & 4.27  & 4.143 & 6.505 & 7.759 & 10.223 & 8.077 & 7.266 & 6.046 & 5.517 \\
 $U_{18}$ &   1.231 & 3.438 & 6.183 & 6.546 & 8.101 & 9.509 & 9.011 & 10.324 & 9.453 & 9.473 & 9.081 & 7.237 & 5.337 & 3.862 & 1.117 & 0.084 & 0.013 & 0     & 0 \\
 $U_{19}$ &   1.861 & 1.95  & 0.932 & 0.991 & 1.15  & 2.118 & 2.271 & 2.785 & 2.5   & 2.302 & 2.971 & 3.458 & 4.603 & 5.653 & 6.467 & 9.109 & 10.856 & 21.894 & 16.129 \\
\end{tabular}%
}
\end{center}
\end{table}

Looking at Table \ref{RAIResults}, one can see that all but two universities can reach the first position but only three of them have a first rank acceptability index greater than 10\%, that is, in the order of their rank acceptability index for the first position, $U_{16}$ ($b^{1}(U_{16})=60.842\%$), $U_{5}$ ($b^{1}(U_{5})=11.409\%$) and $U_{11}$ ($b^{1}(U_{11})=10.098\%$). Considering the tail of this ranking, only six universities cannot be in the last position. Moreover, the universities being most frequently the last are $U_{10}$ ($b^{19}(U_{10})=34.575\%$) followed by $U_{19}$ ($b^{19}(U_{19})=16.129\%)$, $U_{12}$ ($b^{19}(U_{12})=15.236\%$), and $U_{3}$ ($b^{19}(U_{3})=12.983\%$). All other universities have a frequency lower than $9\%$ of being the last. 

\begin{table}[!h]
\begin{center}
\caption{Pairwise winning indices for the considered universities}\label{PWIResults}
\resizebox{1\textwidth}{!}{
 \begin{tabular}{l|ccccccccccccccccccc}
 &  $U_{1}$ &  $U_{2}$ &  $U_{3}$ &  $U_{4}$ &  $U_{5}$ &  $U_{6}$ &  $U_{7}$ &  $U_{8}$ &  $U_{9}$ &  $U_{10}$ &  $U_{11}$ &  $U_{12}$ &  $U_{13}$ &  $U_{14}$ &  $U_{15}$ &  $U_{16}$ &  $U_{17}$ &  $U_{18}$ &  $U_{19}$ \\
\hline
\hline
    $U_{1}$  &0     & 52.88 & 87.76 & 73.891 & 47.394 & 63.221 & 61.815 & 67.549 & 58.72 & 89.058 & 40.502 & 94.58 & 45.355 & 41.761 & 60.788 & 2.67  & 65.033 & 52.922 & 78.926 \\
    $U_{2}$  &47.119 & 0     & 96.115 & 77.306 & 41.156 & 69.175 & 56.098 & 62.675 & 61.247 & 94.356 & 35.04 & 92.351 & 40.654 & 38.117 & 60.672 & 3.088 & 68.563 & 51.892 & 83.673 \\
    $U_{3}$  &12.24 & 3.885 & 0     & 35.834 & 7.502 & 19.641 & 17.501 & 18.88 & 7.237 & 65.397 & 11.531 & 40.937 & 6.422 & 8.678 & 8.695 & 1.807 & 26.983 & 5.179 & 46.993 \\
    $U_{4}$  &26.109 & 22.694 & 64.166 & 0     & 17.139 & 25.031 & 29.685 & 33.314 & 21.757 & 69.916 & 24.762 & 56.472 & 23.987 & 19.271 & 24.485 & 6.669 & 30.481 & 16.775 & 64.76 \\
    $U_{5}$  &52.606 & 58.844 & 92.498 & 82.861 & 0     & 75.625 & 63.038 & 64.312 & 69.535 & 99.71 & 43.33 & 84.363 & 47.921 & 48.795 & 68.406 & 17.957 & 69.914 & 59.988 & 89.043 \\
    $U_{6}$  &36.779 & 30.825 & 80.359 & 74.969 & 24.375 & 0     & 42.531 & 47.458 & 33.406 & 97.501 & 32.581 & 68.821 & 34.679 & 25.95 & 37.342 & 8.135 & 57.689 & 26.147 & 88.518 \\
    $U_{7}$  &38.185 & 43.902 & 82.499 & 70.315 & 36.962 & 57.469 & 0     & 54.937 & 47.14 & 95.878 & 33.907 & 74.925 & 39.099 & 26.674 & 48.793 & 11.583 & 57.873 & 42.855 & 82.077 \\
    $U_{8}$  &32.45 & 37.324 & 81.12 & 66.686 & 35.688 & 52.542 & 45.063 & 0     & 43.95 & 83.903 & 34.362 & 86.365 & 38.132 & 31.024 & 43.796 & 7.345 & 58.468 & 37.233 & 70.393 \\
    $U_{9}$  &41.279 & 38.752 & 92.763 & 78.243 & 30.465 & 66.594 & 52.86 & 56.049 & 0     & 98.008 & 35.124 & 80.173 & 37.753 & 35.376 & 51.255 & 10.723 & 63.42 & 41.508 & 87.061 \\
    $U_{10}$  &10.942 & 5.644 & 34.603 & 30.084 & 0.29  & 2.499 & 4.122 & 16.097 & 1.992 & 0     & 12.713 & 32.6  & 12.026 & 4.647 & 1.612 & 3.194 & 23.716 & 1.945 & 33.74 \\
    $U_{11}$  &59.498 & 64.96 & 88.469 & 75.238 & 56.67 & 67.419 & 66.093 & 65.638 & 64.876 & 87.287 & 0     & 95.279 & 54.991 & 55.243 & 65.079 & 12.737 & 71.528 & 62.203 & 79.975 \\
    $U_{12}$  &5.42  & 7.649 & 59.062 & 43.528 & 15.637 & 31.179 & 25.075 & 13.635 & 19.827 & 67.4  & 4.721 & 0     & 2.67  & 13.532 & 17.485 & 0.309 & 31.26 & 16.342 & 52.378 \\
    $U_{13}$  &54.645 & 59.346 & 93.578 & 76.013 & 52.079 & 65.321 & 60.901 & 61.868 & 62.247 & 87.974 & 45.009 & 97.33 & 0     & 50.715 & 61.209 & 5.42  & 73.457 & 58.617 & 79.244 \\
    $U_{14}$  &58.239 & 61.883 & 91.322 & 80.729 & 51.204 & 74.05 & 73.326 & 68.976 & 64.624 & 95.353 & 44.757 & 86.468 & 49.285 & 0     & 68.872 & 10.233 & 67.769 & 59.295 & 86.586 \\
    $U_{15}$  &39.212 & 39.328 & 91.305 & 75.515 & 31.594 & 62.658 & 51.206 & 56.204 & 48.745 & 98.388 & 34.921 & 82.515 & 38.791 & 31.128 & 0     & 7.6   & 63.462 & 38.995 & 82.695 \\
    $U_{16}$  &97.33 & 96.912 & 98.193 & 93.331 & 82.043 & 91.865 & 88.417 & 92.655 & 89.277 & 96.806 & 87.263 & 99.691 & 94.58 & 89.767 & 92.4  & 0     & 97.547 & 89.933 & 92.688 \\
    $U_{17}$  &34.967 & 31.437 & 73.017 & 69.519 & 30.086 & 42.311 & 42.127 & 41.532 & 36.58 & 76.284 & 28.472 & 68.74 & 26.543 & 32.231 & 36.538 & 2.453 & 0     & 30.261 & 69.217 \\
    $U_{18}$  &47.078 & 48.108 & 94.821 & 83.225 & 40.012 & 73.853 & 57.145 & 62.767 & 58.492 & 98.055 & 37.797 & 83.658 & 41.383 & 40.705 & 61.005 & 10.067 & 69.739 & 0     & 87.711 \\
    $U_{19}$  &21.074 & 16.327 & 53.007 & 35.24 & 10.957 & 11.482 & 17.923 & 29.607 & 12.939 & 66.26 & 20.025 & 47.622 & 20.756 & 13.414 & 17.305 & 7.312 & 30.783 & 12.289 & 0 \\
\end{tabular}%
}
\end{center}
\end{table}

Analogously, to pairwise compare the considered universities, we computed the pairwise winning indices shown in Table \ref{PWIResults}. Considering the three universities being most frequently in the first position, that are $U_{16}$, $U_{5}$ and $U_{11}$, one can observe that $U_{16}$ is preferred to $U_{5}$ and $U_{11}$ with frequencies of $82.043\%$ and $87.263\%$, respectively. Moreover, it is preferred to all other universities with a frequency at least equal to the 88.417\%. This enforces the conviction that $U_{16}$ is the best among the universities at hand.

On the basis of the rank acceptability indices, as done also in \cite{AngilellaEtAl2018}, for each university we compute the best and the worst possible positions (Table \ref{BW}) and the most frequent ones (Table \ref{MF}). In this way, we have a clearer view of the goodness of the considered universities. Almost all universities have as best and worst reached positions the first and the last one. Anyway, the data in Table \ref{MF} give an idea on the stability of the results of each university. Indeed, $U_{16}$, $U_{5}$ and $U_{11}$ have their highest rank acceptability indices values in correspondence of the first three ranking positions. Looking at the opposite side of the ranking, $U_{10}$ and $U_{19}$ are confirmed as the worst universities since they have the highest values for their rank acceptability indices in correspondence of the last three positions. To have a look at similar results on the whole set of universities, the interested reader is deferred to the supplementary material downloadable clicking on the following link: \href{http://http://www.antoniocorrente.it/wwwsn/images/allegati_articoli/ChLevelDependentNARORSMAA.xlsx}{supplementary results}.

\begin{table}[!h]
\begin{center}
\caption{Best, worst and most frequent positions}\label{BWMFResults}
\subtable[Best and worst positions\label{BW}]{%
\resizebox{0.3\textwidth}{!}{
 \begin{tabular}{l|cccc}    
 & Best & $b^{West}(\cdot)$ & Worst & $b^{Worst}(\cdot)$ \\
\hline
\hline
$U_{1}$ &1     & 0.261 & 19    & 0.263 \\
    $U_{2}$ &2     & 0.221 & 16    & 0.04 \\
    $U_{3}$ &1     & 0.001 & 19    & 12.983 \\
    $U_{4}$ &1     & 0.45  & 19    & 8.828 \\
    $U_{5}$ &1     & 11.409 & 18    & 0.409 \\
    $U_{6}$ &1     & 0.156 & 19    & 0.031 \\
    $U_{7}$ &1     & 4.925 & 19    & 0.788 \\
    $U_{8}$ &1     & 2.165 & 19    & 1.942 \\
    $U_{9}$ &1     & 1.125 & 17    & 0.106 \\
    $U_{10}$ &1     & 0.002 & 19    & 34.575 \\
    $U_{11}$ &1     & 10.098 & 19    & 3.177 \\
    $U_{12}$ &3     & 0.06  & 19    & 15.236 \\
    $U_{13}$ &1     & 2.881 & 19    & 0.228 \\
    $U_{14}$ &1     & 1.478 & 19    & 0.302 \\
    $U_{15}$ &1     & 0.006 & 17    & 0.049 \\
    $U_{16}$ &1     & 60.842 & 15    & 0.125 \\
    $U_{17}$ &1     & 1.109 & 19    & 5.517 \\
    $U_{18}$ &1     & 1.231 & 17    & 0.013 \\
    $U_{19}$ &1     & 1.861 & 19    & 16.129 \\
\end{tabular}%
}
}
\subtable[Most frequent positions\label{MF}]{%
\resizebox{0.4\textwidth}{!}{
\begin{tabular}{l|cccccc}    
 & $high_{1}$ & $b^{high_{1}}(\cdot)$ & $high_{2}$ & $b^{high_{2}}(\cdot)$ & $high_{3}$ & $b^{high_{3}}(\cdot)$ \\
\hline
\hline
$U_{1}$  &4     & 11.567 & 5     & 10.474 & 3     & 10.258 \\
    $U_{2}$  &7     & 13.971 & 8     & 13.33 & 9     & 12.842 \\
    $U_{3}$  &17    & 16.734 & 16    & 16.019 & 18    & 15.598 \\
    $U_{4}$  &17    & 12.042 & 18    & 11.055 & 16    & 9.766 \\
    $U_{5}$  &1     & 11.409 & 3     & 8.489 & 2     & 8.189 \\
    $U_{6}$  &14    & 9.421 & 13    & 8.734 & 10    & 7.523 \\
    $U_{7}$  &15    & 7.919 & 16    & 7.65  & 5     & 7.326 \\
    $U_{8}$  &14    & 7.416 & 13    & 6.939 & 15    & 6.837 \\
    $U_{9}$  &12    & 11.091 & 11    & 10.572 & 13    & 9.294 \\
    $U_{10}$  &19    & 34.575 & 18    & 17.199 & 17    & 14.982 \\
    $U_{11}$  &2     & 15.65 & 3     & 15.175 & 1     & 10.098 \\
    $U_{12}$  &19    & 15.236 & 18    & 14.737 & 17    & 12.956 \\
    $U_{13}$  &2     & 14.07 & 3     & 13.639 & 4     & 9.138 \\
    $U_{14}$  &2     & 14.957 & 3     & 10.974 & 4     & 9.822 \\
    $U_{15}$  &10    & 13.109 & 9     & 12.741 & 8     & 10.36 \\
    $U_{16}$  &1     & 60.842 & 2     & 17.604 & 3     & 7.103 \\
    $U_{17}$  &15    & 10.223 & 16    & 8.077 & 14    & 7.759 \\
    $U_{18}$  &8     & 10.324 & 6     & 9.509 & 10    & 9.473 \\
    $U_{19}$  &18    & 21.894 & 19    & 16.129 & 17    & 10.856 \\
\end{tabular}%
}
}
\end{center}
\end{table}

In some real world applications, the DM could wish a complete ranking of the alternatives under analysis. To summarize the results of the SMAA with a specific attention to the rank acceptability indices, \cite{KadzinskiMichalski2016} provided several aggregation procedures aiming to obtain a unique final ranking. In this case, we shall apply one of them in which each university $U$ will be associated to the following weighted mean of its rank acceptability indices: 

\begin{equation*}
E(U)=-\displaystyle\sum_{s=1}^{19}s\cdot b^{s}(U). 
\end{equation*}
 
\begin{table}[!h]
\begin{center}
\caption{Expected ranking of the universities at hand}\label{ExpectedRanking}
\resizebox{0.3\textwidth}{!}{
 \begin{tabular}{ccc}
    University & $E(\cdot)$ & Rank-position \\
		\hline
		\hline
    $U_{16}$    & -229.302 & 1 \\
    $U_{11}$    & -706.817 & 2\\
    $U_{14}$    & -707.028 & 3\\
    $U_{5}$     & -711.253 & 4\\
    $U_{13}$    & -755.027 & 5\\
    $U_{18}$    & -804.379 & 6\\
    $U_{1}$     & -815.172 & 7\\
    $U_{2}$     & -820.7 & 8\\
    $U_{9}$     & -902.591 & 9\\
    $U_{15}$    & -925.737 & 10\\
    $U_{7}$     & -954.926 & 11\\
    $U_{8}$     & -1014.15 & 12\\
    $U_{6}$     & -1051.94 & 13\\
    $U_{17}$    & -1127.69 & 14\\
    $U_{4}$     & -1322.53 & 15\\
    $U_{19}$    & -1455.68 & 16\\
    $U_{12}$    & -1472.89 & 17\\
    $U_{3}$     & -1554.66 & 18\\
    $U_{10}$    & -1667.53 & 19\\
    \end{tabular}%
}
\end{center}
\end{table}

Looking at the results in Table \ref{ExpectedRanking}, once again we have the proof that $U_{16}$ can be reasonably considered the best among the universities at hand, followed by $U_{11}$ and $U_{14}$, while $U_{10}$ and $U_{3}$ can be fairly considered the worst ones. A careful look at the rank acceptability indices in Table \ref{RAIResults} and to the most frequent positions in Table \ref{MF} would justify the second and third positions attained by $U_{11}$ and $U_{14}$ in place of $U_{5}$. Indeed, even if $U_{5}$ reaches the first position more frequently than the other two universities, $U_{11}$ and $U_{14}$ present high frequencies for the rank acceptability indices corresponding to the second and third positions, while $U_{5}$ has a frequency lower than 10\% for the same positions. Therefore, since the expected ranking takes into account all the frequencies with which an alternative reaches a certain position and not only the first one, $U_{11}$ and $U_{14}$ can be correctly considered better than $U_{5}$. 

%%%%%%%%%%%%%%%%%%%%%%%%%%%%%%%%%%%%%
\section{Conclusions}\label{Concl}%%%
%%%%%%%%%%%%%%%%%%%%%%%%%%%%%%%%%%%%%
Parsimony of models in multiple criteria decision analysis has to be tempered by the possibility to permit a rich interaction with the decision maker (DM) and the consideration of robustness concerns. In this perspective we presented a decision model that has the following interesting properties:
\begin{itemize}
	\item interaction between criteria is taken into consideration,
	\item importance and interaction of criteria can change from one level to the other of the evaluations taken by considered alternatives on criteria at hand,
	\item the parameters of the model are obtained by starting from some preference information supplied by the decision makers in terms that are relatively easily understandable by them such as: 
		\begin{itemize}
		\item pairwise comparisons of alternatives, 
		\item comparisons of criteria in terms of their importance,
		\item nature of interaction between criteria (synergy or  redundancy) and comparisons related to intensity of the interaction between pairs of criteria;
			\end{itemize}
	\item consideration of the plurality of instances of the decision model compatible with the preferences provided by the DM so that, not only necessary preferences (valid for all compatible models) are distinguished from possible preferences (valid for at least one compatible model), but even a probabilistic evaluation of a preference of an alternative over another or the ranking position of the alternatives is given to the DM. 
\end{itemize}
All this is obtained by coupling: 
\begin{itemize}
	\item the level dependent Choquet integral, which permits to take into account importance and interaction of criteria that can change from one level to the other of the  evaluations of criteria,
	\item the robust ordinal regression and the stochastic multicriteria acceptability analysis, which permit to take into account the plurality of instances of the preference model compatible with the preference information supplied by the decision makers. 
\end{itemize}

We considered two possible representations of the changes of importance and interaction between criteria from one level to another of the evaluations:
\begin{itemize}
\item the model of interval level dependent capacity, in which the range of considered evaluations is split in a certain number of subintervals and the importance and interaction of criteria changes from one subinterval to the others, but remain the same in each subinterval,
\item the model of the piecewise linear level dependent capacities, in which the importance and the interaction of criteria change with continuity in the whole domain without the jumps that are present in the previous model.
\end{itemize}

Both models permit to apply Robust Ordinal Regression and Stochastic Multicriteria Acceptability Analsysis by solving LP problems and using the Hit and Run algorithm, which is a well known algorithm for sampling uniformly points in convex regions of the preference parameters space. 

We believe that the good properties of the methodology we have presented can be very useful in many real life complex decision problems in the main domains of interest such as sustainable development \cite{AngilellaEtAl2018}, wellbeing and happiness economics \cite{GrecoEtAl2019,greco2019sigma}, ranking of universities \cite{CorrenteGrecoSlowinski2019}, country competitiveness \cite{CorrenteEtAl2018}, and so on.

%%%%%%%%%%%%%%%%%%%%%%%%%%%%%
\section*{Acknowledgments}%%%
%%%%%%%%%%%%%%%%%%%%%%%%%%%%%
The second and the third authors wish to acknowledge funding by the FIR of the University of Catania BCAEA3 ``New developments in Multiple Criteria Decision Aiding (MCDA) and their application to territorial competitiveness''

%%%%%%%%%%%%%%%%%%%%%%%%
\section*{References}%%%
%%%%%%%%%%%%%%%%%%%%%%%%
%%%%%%%%%%%%%%%%%%%%%%%%%%%%%%%%%%%
\bibliographystyle{plain}%%%%%%%%%%
\bibliography{Full_bibliography}%%%

\end{document}